\documentclass[fleqn,final]{zzz}
\usepackage{mathrsfs}
\usepackage{amsmath}
\usepackage{graphicx}
\usepackage[pdftex]{hyperref}
\usepackage[normalem]{ulem}
\usepackage{color}
\usepackage{soul}
\usepackage{xparse}
\usepackage{chngcntr}
\usepackage{apptools}
\usepackage{tikz}
\usepackage{rotating}
\usepackage{longtable}
\usepackage{comment}

\hypersetup{
colorlinks = true,
urlcolor = blue,
linkcolor = black,
}

\newcommand{\Whyp}[5]{\,\mbox{}_{#1}W_{#2}\!\left({#3};{#4};{#5}\right)}
\newcommand{\Wphyp}[6]{\,\sideset{_{#1}^{\phantom{\mid}}}{_{#2}^{#3}}
{\mathop{W}}\!\left({#4};{#5};{#6}\right)}
\newcommand{\qhyp}[5]{\,\mbox{}_{#1}\phi_{#2}\!\left(
\genfrac{}{}{0pt}{}{#3}{#4};#5\right)}
\newcommand{\qphyp}[6]{\,\sideset{_{#1}^{\phantom{\mid}}}{_{#2}^{#3}}
{\mathop{\phi}}\!\left(\genfrac{}{}{0pt}{}{#4}{#5};#6\right)}

\NewDocumentCommand{\qfrac}{smm}{%
 \dfrac{\IfBooleanT{#1}{\vphantom{\big|}}#2}{\mathstrut #3}%
}

\def\cprime{$'$}

\newtheorem{thm}[lemma]{Theorem}
\newtheorem{cor}[lemma]{Corollary}
\newtheorem{rem}[lemma]{Remark}

\newtheorem{lem}[lemma]{Lemma}

\makeatletter
\def\eqnarray{\stepcounter{equation}\let\@currentlabel=\theequation
\global\@eqnswtrue
\tabskip\@centering\let\\=\@eqncr
$$\halign to \displaywidth\bgroup\hfil\global\@eqcnt\z@
 $\displaystyle\tabskip\z@{##}$&\global\@eqcnt\@ne
 \hfil$\displaystyle{{}##{}}$\hfil
 &\global\@eqcnt\tw@ $\displaystyle{##}$\hfil
 \tabskip\@centering&\llap{##}\tabskip\z@\cr}

\def\endeqnarray{\@@eqncr\egroup
 \global\advance\c@equation\m@ne$$\global\@ignoretrue}

\def\@yeqncr{\@ifnextchar [{\@xeqncr}{\@xeqncr[5pt]}}
\makeatother

\newcommand{\N}{{\mathbb N}}

\newcommand{\R}{{\mathbb R}}
\newcommand{\Z}{{\mathbb Z}}
\newcommand{\CC}{{\mathbb C}}
\newcommand{\expe}{{\mathrm e}}
\newcommand{\CCast}{{\mathbb C}^\ast}
\newcommand{\CCdag}{{\mathbb C}^\dag}
\newcommand{\SSS}{{\mathcal S}}

\numberwithin{equation}{section}
\numberwithin{corollary}{section}
\numberwithin{remark}{section}
\numberwithin{theorem}{section}
\numberwithin{lemma}{section}
\numberwithin{definition}{section}
\definecolor{darkgreen}{rgb}{0.0, 0.21, 0.06}

\begin{document}

\renewcommand{\PaperNumber}{***}

\FirstPageHeading

\ShortArticleName{Terminating symmetric basic hypergeometric 
representations and transformations}

\ArticleName{Terminating representations, 
transformations and summations for the $q$ and $q^{-1}$-symmetric 
subfamilies of the 
Askey--Wilson polynomials}

% Names of the authors for the title of the paper
\Author{Howard S. Cohl\,$^\ast\!\!\ $, Roberto 
S. Costas-Santos\,$^\dag\!$ and Linus Ge\,$^{\ddag}\!\!$}

\AuthorNameForHeading{H.~S.~Cohl, R.~S.~Costas-Santos and L.~Ge}
\Address{$^\ast$ Applied and Computational 
Mathematics Division, National Institute of Standards 
and Technology, Mission Viejo, CA 92694, USA
%Address of First Author, Country
\URLaddressD{
\href{http://www.nist.gov/itl/math/msg/howard-s-cohl.cfm}
{http://www.nist.gov/itl/math/msg/howard-s-cohl.cfm}
}
} % Address of First Author
\EmailD{howard.cohl@nist.gov} % E-mail address of First Author

\Address{$^\dag$ Department of Quantitative 
Methods, Universidad Loyola Andaluc\'ia, 
E-41704 Seville, Spain
} 
% Address of First Author
\URLaddressD{
\href{http://www.rscosan.com}
{http://www.rscosan.com}
}
\EmailD{rscosa@gmail.com} % E-mail address of First Author

\Address{$^{\ddag}$ Department of Mathematics,
The Ohio State University, Columbus, OH 43210, USA
% Address of First Author
}
\EmailD{ge.409@buckeyemail.osu.edu} % E-mail address of First Author

\ArticleDates{Received 30 June 2020 in final form ????; 
Published online ????}

\Abstract{In this article, we exhaustively explore the 
terminating basic hypergeometric representations and transformations 
of the $q$ and $q^{-1}$-symmetric subfamilies of the 
Askey--Wilson polynomials. These subfamilies are obtained by 
repeatedly setting one of the free parameters (not $q$) equal 
to zero until no parameters are left. 
These subfamilies (and their $q^{-1}$ counterparts) are the
continuous dual $q$-Hahn, Al-Salam--Chihara, continuous big $q$-Hermite,
and the continuous $q$-Hermite polynomials. From the terminating basic
hypergeometric representations of these polynomials, and due to
symmetry in their free parameters, we are able to exhaustively explore
the terminating basic hypergeometric transformation formulas which
these polynomials satisfy. We then study the terminating transformation structure which are implied by the terminating representations of these polynomials. We conclude by describing the symmetry group structure of the $q$-Askey scheme.
}

%``Symmetry, Integrability and Geometry: Methods and Applications''.}

%Please type here List of Keywords for your article separated 
%by semicolon.
% Keywords:

\Classification{33D15, 33D45}
%{??????} % e.g. 35A30; 81Q05
%For 2010 Mathematics Subject Classification see
%http://www.ams.org/mathscinet/msc/msc2010.html

%\tableofcontents
%\addtocontents{toc}{\protect\color{black}}

%%%%%%%%%%%%%%%%%%%%%%%%%%%%%%%%%%%%%%%%%%%%%
\section{Introduction} 
\label{Introduction}
%%%%%%%%%%%%%%%%%%%%%%%%%%%%%%%%%%%%%%%%%%%%%

The work contained in this paper is concerned with transformation 
properties of the $q$ and $q^{-1}$-symmetric subfamilies of the Askey--Wilson polynomials 
(see e.g., \cite[Section 14.1]{Koekoeketal}). The polynomials in these symmetric subfamilies are the continuous dual $q$-Hahn polynomials, the Al-Salam--Chihara polynomials, the continuous big $q$-Hermite polynomials, the continuous $q$-Hermite polynomials and their $q^{-1}$-symmetric counterparts. By constructing all terminating basic hypergeometric representations of these symmetric subfamilies we are able to study the set of transformations between the specific terminating basic hypergeometric functions which appear as representations. By examining this structure we are able to study the symmetry group structure of these polynomials which has not as of yet been fully studied in the literature.

%%%%%%%%%%%%%%%%%%%%%%%%%%%%%%%%%%%%%%%%%%%%%
%\section{Preliminaries}%\label{Preliminaries}
%%%%%%%%%%%%%%%%%%%%%%%%%%%%%%%%%%%%%%%%%%%%%
We adopt the following set 
notations:~$\N_0:=\{0\}\cup
\N=\{0, 1, 2, \ldots \}$, and we use the 
sets $\Z$, $\R$, $\CC$ which represent 
the integers, real numbers and complex 
numbers respectively, 
$\CCast:=\CC\setminus\{0\}$,
and $\CCdag:=\CCast\setminus
\{z\in\CC: |z|=1\}$.
We also adopt the following notation 
and conventions.

Recall the notion of a multiset which extends 
the definition of a set where the multiplicity 
of elements is allowed. This notion
becomes important for hypergeometric functions, 
where numerator and denominator parameter 
entries may be identical. We adopt the 
following conventions for succinctly writing 
elements of multisets.
For a multiset of cardinality $k\in\mathbb N_0$, ${\bf b}:=\{b_1, \ldots,b_k\}$, 
$b_k\in{\mathbb C^\ast}$, 
$k\in\mathbb N$,
${\bf b}(x):=\{b_1(x), \ldots,b_k(x)\},$
and the reciprocal notation
${\bf b}^{-1}:=
\bigl\{b_1^{-1}, \ldots,b_k^{-1}\bigr\}$.
We will use the set notation 
${\mathbf b}_{[k]}:=
{\mathbf b}\setminus\{b_k\}$.
Let ${\bf a}:=\{a_1, a_2, a_3, a_4\}$, 
$b, a_k\in\mathbb C$, $k=1, 2, 3, 4$. 
Define
${\bf a}+b:=\{a_1+b, a_2+b, a_3+b, a_4+b\}$,
$a_{12}:=a_1a_2$,
$a_{13}:=a_1a_3$,
$a_{23}:=a_2a_3$,
$a_{123}:=a_1a_2a_3$,
$a_{1234}:=a_1a_2a_3a_4$, etc. 
Consider a sequence of complex numbers 
$\{a_k\}$, $k\in\mathbb N_0$. 
Furthermore, let $s,r\in\mathbb N_0$, 
with $s<r$. 
Then we assume that the the empty sum 
vanishes and the 
empty product 
is unity, namely
$\sum_{k=r}^s a_k=0, 
\quad \prod_{k=r}^s a_k=1$.
Throughout this paper we 
adopt the following conventions 
for succinctly writing certain 
sets and subsets.
To indicate sequential positive 
and negative elements, we write
$\pm a:=\{a,-a\}$.
We also adopt an analogous notation
for $z\in\CC^\ast$, $a\in\CC$,
$z^{\pm a}:=\{z^{a},z^{-a}\}$, e.g.,$z^{\pm}:=\{z,z^{-1}\}$.
%If $\pm$ appears in an expression, 
%but not in 
%a list, it is to be treated as normal.
In the same vein, consider a finite 
sequence $f_s\in\mathbb C$ with 
$s\in{\mathcal S}\subset \mathbb N$.
Then, the notation
$\{f_s\}$
represents the sequence of all complex 
numbers $f_s$ such that 
$s\in\SSS$.
Furthermore, consider some $p\in\SSS$, then 
the notation $\{f_s\}_{s\ne p}$ represents 
the sequence of all complex numbers $f_s$ 
such that $s\in\SSS\!\setminus\!\{p\}$.
Also for $n\le0$, we take
$\{a_1, \ldots,a_n\}:=\emptyset.$

Consider $q\in\mathbb C^\ast$ such that 
$|q|\ne 1$.
Define the sets 
$\Omega_q^n:=\{q^{-k}:n,k\in\mathbb 
N_0,~0\le k\le n-1\}$,
$\Omega_q:=\Omega_q^\infty=
\{q^{-k}:k\in\mathbb N_0\}$.
In order to obtain our derived identities, 
we rely on properties of the $q$-Pochhammer 
symbol ($q$-shifted factorial). 
For any $n\in \mathbb N_0$, $a,q\in \mathbb C$, 
%the Pochhammer symbol, and 
the $q$-Pochhammer symbol is defined as
\begin{equation}
%\hspace{-1cm}
(a;q)_n:=(1-a)(1-aq)\cdots(1-aq^{n-1}).
\label{poch.id:1}
\end{equation}
One may also define
\begin{equation}
(a;q)_\infty:=\prod_{n=0}^\infty (1-aq^{n}),\label{poch.id:2}
\end{equation}
where $|q|<1$.
Furthermore, define 
%for all $a,b\in\mathbb C$, 
%\[
%(a)_b:=\dfrac{\Gamma(a+b)}{\Gamma(a)},
%=\dfrac{(a)_\infty}{(a+b)_\infty},
%\]
%where $a+b\not\in-\mathbb N_0$, and 
\[
(a;q)_b:=\frac{(a;q)_\infty}
{(a q^b;q)_\infty}.
\]
where $a q^b\not \in \Omega_q$.
We will also use the common notational product convention
\begin{eqnarray}
%&&\hspace{-8cm}(a_1, \ldots,a_k)_b:=(a_1)_b\cdots(a_k)_b,\nonumber\\
&&\hspace{-8cm}(a_1, \ldots,a_k;q)_b:=(a_1;q)_b\cdots(a_k;q)_b,\nonumber
\end{eqnarray}
where $b\in{\mathbb C}\cup\{\infty\}$.
%$q$-Pochhammer symbols
%are used in $q$-special functions.
%We define the $q$-factorial as 
%\cite[(1.2.44)]{GaspRah}
%\[
%[0]_q!:=1,\ [n]_q!:=[1]_q[2]_q\cdots [n]_q,\quad n\in\N,
%\]
%where the $q$-number is defined as \cite[(1.8.1)]{Koekoeketal}
%\[
%[z]_q:=\frac{1-q^z}{1-q},\quad z\in \mathbb C.
%\]
%Note that $[n]_q!=(q;q)_n / (1-q)^n$, $n\in\mathbb N_0$.

The following properties for the $q$-Pochhammer 
symbol can be found in Koekoek et al. 
\cite [(1.8.7), (1.8.10-11), (1.8.14), (1.8.19), 
(1.8.21-22)]{Koekoeketal}, namely for appropriate 
values of $q,a\in\CCast$ and $n,k\in\mathbb N_0$:
\begin{eqnarray}
\label{poch.id:3} &&\hspace{-6.5cm}
%(a;q)_n=q^{\binom n 2}(-a)^n(a^{-1};q^{-1})_n,
(a;q^{-1})_n=(a^{-1};q)_n(-a)^nq^{-\binom{n}{2}},
\\[2mm] 
\label{poch.id:4}
&&\hspace{-6.5cm}(a;q)_{n+k}=(a;q)_k(aq^k;q)_n 
= (a;q)_n(aq^n;q)_k,\\[2mm] 
\label{poch.id:5}&&\hspace{-6.5cm} (a;q)_n=(q^{1-n}/a;q)_n(-a)^nq^{\binom{n}{2}},\\[2mm]
\label{poch.id:6}&&\hspace{-6.5cm}(aq^{-n};q)_{k}=q^{-nk}
\frac{(q/a;q)_n}{(q^{1-k}/a;q)_n}(a;q)_k,\\[2mm]
\label{poch.id:7}&&\hspace{-6.5cm}(a;q)_{2n}=(a,aq;q^2)_n=(\pm\sqrt{a},\pm\sqrt{qa};q)_n
\end{eqnarray}
\noindent Observe that by using (\ref{poch.id:1}) and 
\cite[(1.8.22)]{Koekoeketal}, one obtains
\begin{eqnarray}
&&\hspace{-7.3cm}(aq^n;q)_n=\frac{(\pm 
\sqrt{a},\pm \sqrt{aq};q)_n}{(a;q)_n},\quad 
a\not\in\Omega_q^n. 
\label{poch.id:9}
\end{eqnarray}
Note the equivalent representation of \eqref{poch.id:3} which is
very useful for obtaining limits which we
often need is
\begin{equation*}
%\hspace{-7.2cm}
a^n\left(\frac{x}{a};q\right)_n=
q^{\binom{n}{2}}(-x)^n \left(\frac{a}{x};q^{-1}\right)_n,
\end{equation*}
therefore
\begin{equation*}
%\hspace{-5.3cm}
\lim_{a\to0}\,a^n\left(\frac{x}{a};q\right)_n
=
\lim_{b\to\infty}\,\frac{1}{b^n}\left(xb;q\right)_n
=
q^{\binom{n}{2}}(-x)^n.
\end{equation*}
From \eqref{poch.id:5}, another useful limit representation is
\begin{equation}
%\hspace{-9.8cm} 
\lim_{\lambda\to\infty}\frac{(a\lambda;q)_n}{(b\lambda;q)_n}
=\left(\frac{a}{b}\right)^n.
\label{lambdamu}
\end{equation}
Other useful identities for simplifying $q$-Pochhammer symbol expressions
are for $a,b\in\mathbb C^{\ast}$, $n\in\mathbb N_0$, 
are
\begin{eqnarray}
&&\hspace{-7.9cm}(q^{-n}a;q)_{n}=q^{-\binom{n}{2}}\left(\frac{-a}{q}\right)^n
\left(\frac{q}{a};q\right)_n,\\
&&\hspace{-7.9cm}(q^{-n}a;q)_{2n}=q^{-\binom{n}{2}}\left(\frac{-a}{q}\right)^n(a;q)_n
\left(\frac{q}{a};q\right)_n,\\
&&\hspace{-7.9cm}\frac{(q^{-n}a;q)_n}{(q^{-n}b;q)_n}=
\left(\frac{a}{b}\right)^n\frac{(\frac{q}{a};q)_n}{\left(\frac{q}{b};q\right)_n},\\
&&\hspace{-7.9cm}\frac{(q^{-2n}a;q)_n}{(q^{-2n}b;q)_n}=
\left(\frac{a}{b}\right)^n\frac{
\left(\frac{q}{b};q\right)_{n}
\left(\frac{q}{a};q\right)_{2n}
}{
\left(\frac{q}{a};q\right)_{n}
\left(\frac{q}{b};q\right)_{2n}
}.
\end{eqnarray}
The basic hypergeometric series, which we 
will often use, is defined for
$q,z\in\CCast$ such that $|q|,|z|<1$, $s,r\in\mathbb N_0$, 
$b_j\not\in\Omega_q$, $j=1, \ldots,s$, as
\cite[(1.10.1)]{Koekoeketal}
\begin{equation}
\qhyp{r}{s}{a_1, \ldots,a_r}
{b_1, \ldots,b_s}
{q,z}
:=\sum_{k=0}^\infty
\frac{(a_1, \ldots,a_r;q)_k}
{(q,b_1, \ldots,b_s;q)_k}
\left((-1)^kq^{\binom k2}\right)^{1+s-r}
z^k.
\label{2.11}
\end{equation}
In the case where either $r=0$ or $s=0$, then we write a dash in place of the non-existent numerator or denominator parameters.
Note that we refer to a basic hypergeometric
series as {\it $\ell$-balanced} if
$q^\ell a_1\cdots a_r=b_1\cdots b_s$, 
and {\it balanced} if $\ell=1$.
A basic hypergeometric series ${}_{r+1}\phi_r$ is {\it well-poised} if 
the parameters satisfy the relations
\[
qa_1=b_1a_2=b_2a_3=\cdots=b_ra_{r+1}.
\]
It is {\it very-well-poised} if in addition, 
$\{a_2,a_3\}=\pm q\sqrt{a_1}$.
{\it Terminating} basic hypergeometric series 
which appear in basic hypergeometric orthogonal
polynomials, are defined as
\begin{equation}
\qhyp{r}{s}{q^{-n},a_1, \ldots,a_{r-1}}
{b_1, \ldots,b_s}{q,z}:=\sum_{k=0}^n
\frac{(q^{-n},a_1, \ldots,a_{r-1};q)_k}{(q,b_1, \ldots,b_s;q)_k}
\left((-1)^kq^{\binom k2}\right)^{1+s-r}z^k,
\label{2.12}
\end{equation}
where $b_j\not\in\Omega_q^n$, $j=1, \ldots,s$.
We will use the following alternative notation due to van de Bult and Rains \cite[p.~4]{vandeBultRains09}
for
basic hypergeometric series with zero parameter
entries, namely for $p\in\N_0$, one has
\begin{equation}\label{topzero} 
\qphyp{r+1}{s}{-p}{a_1, \ldots,a_{r+1}}
{b_1, \ldots,b_s}{q,z}
:=\qhyp{r+1+p}{s}{a_1, \ldots,
a_{r+1},\overbrace{0, \ldots,0}^{p}}
{b_1, \ldots,b_s}{q,z},
\end{equation}
\begin{equation}\label{botzero}
\qphyp{r+1}{s}{p}{a_1, \ldots,a_{r+1}}
{b_1, \ldots,b_s}{q,z}
:=\qhyp{r+1}{s+p}
{a_1, \ldots,a_{r+1}}
{b_1, \ldots,b_s,
\underbrace{0, \ldots,0}_{p}}{q,z},
\end{equation}
where $b_1, \ldots,b_s\not
\in\Omega_q^n\cup\{0\}$, and
\[
{}_{r+1}\phi_s^{0}
%{}_{r+1}\phi_{s,0}
={}_{r+1}\phi_s.
\]
Define the very-well-poised 
basic hypergeometric series
${}_{r+1}W_r$ \cite[(2.1.11)]{GaspRah}
\begin{equation}
\label{rpWr}
{}_{r+1}W_r(a;a_4, \ldots,a_{r+1};q,z)
:=\qhyp{r+1}{r}{\pm q\sqrt{a},a,a_4, \ldots,a_{r+1}}
{\pm \sqrt{a},\frac{qa}{a_4}, \ldots,\frac{qa}{a_{r+1}}}{q,z},
\end{equation} 
where $\sqrt{a},\frac{qa}{a_4}, \ldots,\frac{qa}{a_{r+1}}\not\in\Omega_q$. 
When the very-well-poised basic hypergeometric
series is terminating, then one has
\begin{equation}\label{eq:2.13}
{}_{r+1}W_r\left(a;q^{-n},a_5, \ldots,a_{r+1};q,z\right)
=\qhyp{r+1}{r}{q^{-n},\pm q\sqrt{a}, a, a_5, \ldots,a_{r+1}}
{\pm \sqrt{a},q^{n+1}a,\frac{qa}{a_5}, \ldots,\frac{qa}{a_{r+1}}}
{q,z},
\end{equation}
where $\sqrt{a},\frac{qa}{a_5}, \ldots,\frac{qa}{a_{r+1}}\not\in
\Omega_q^n\cup\{0\}$.
We adopt the following van de Bult and Rains extensions of terminating very-well-poised basic hypergeometric series, namely 
\begin{equation}\label{topWzero} 
\Wphyp{r+1}{r}{p}{a}
{q^{-n},a_5, \ldots,a_{r+1}}{q,z}:=\qphyp{r+1}{r}{p}{q^{-n}, 
\pm q\sqrt{a}, a, a_5, \ldots,a_{r+1}
}
{\pm\sqrt{a},q^{n+1}a,\frac{qa}{a_5}, \ldots,\frac{qa}{a_{r+1}}}{q,z},
\end{equation}
where $\sqrt{a},\frac{qa}{a_5}, \ldots,\frac{qa}{a_{r+1}}\not
\in\Omega_q^n\cup\{0\}$, and
${}_{r+1}W_{r}^{0}={}_{r+1}W_r$.

We will often use (frequently without mentioning) the 
limit transition 
formulas for basic hypergeometric series which can be found in \cite[(1.10.3-5)]{Koekoeketal}.
Using these limit-transition formulas for basic hypergeometric 
series, and using the definitions of the extensions to very-well-poised 
hypergeometric series \eqref{topWzero}, one can obtain the 
following limit transition formulas.
One may construct the following $\lambda\to\infty$ and the corresponding $\epsilon\to0$ limit transition formulas
by replacing $\lambda\mapsto\epsilon^{-1}$. 
\begin{lem}
\label{leminftyWhyp}
Let $p\in{\mathbb Z}$, $z\in{\mathbb C}^\ast$, $a,a_k,q\in\CCast$, 
such that $|q\ne1$, 
$\sqrt{a},
\frac{qa}{a_5}, \ldots,\frac{qa}
{a_{r+1}}\not\in\Omega_q$,
$k=5, \ldots,r+1$. 
Then
\begin{eqnarray}
&&\hspace{-0.5cm}\label{Whypliminfty1}\lim_{\lambda\to\infty}\!
\Wphyp{r+1}{r}{p}{a}{q^{-n},a_5, \ldots,a_r,\lambda a_{r+1}}{q,\frac{z}{\lambda}}
=\Wphyp{r}{r-1}{p+1}{a}{q^{-n},a_5, \ldots,a_r}{q,a_{r+1}z},\\
&&\hspace{-0.5cm}\label{Whypliminfty3}\lim_{\lambda\to\infty}\!
\Wphyp{r+1}{r}{p}{a}{q^{-n},a_5, \ldots,a_r,\frac{a_{r+1}}{\lambda}}{q,\lambda z}
=\Wphyp{r}{r-1}{p-1}{a}{q^{-n},a_5, \ldots,a_r}{q,\frac{a_{r+1}z}{qa}},\\
&&\hspace{-0.5cm}\label{Whyplimzero1}\lim_{\epsilon\to0}
\Wphyp{r+1}{r}{p}{a}
{q^{-n},a_5, \ldots,a_r,\epsilon a_{r+1}}
{q,\frac{z}{\epsilon}}
=\Wphyp{r}{r-1}{p-1}{a}{q^{-n},a_5, \ldots,a_r}{q,
\frac{a_{r+1}z}{qa}},\\
&&\hspace{-0.5cm}\label{Whyplimzero3}\lim_{\epsilon\to0}
\Wphyp{r+1}{r}{p}{a}
{q^{-n},a_5, \ldots,a_r,\frac{a_{r+1}}{\epsilon}}
{q,\epsilon z}
=\Wphyp{r}{r-1}{p+1}{a}{q^{-n},a_5, \ldots,a_r}{q,a_{r+1}z}.
\end{eqnarray}
\end{lem}
\begin{proof}
These follow directly from the definitions \eqref{topWzero} and the limit transitions \cite[(1.10.3-5)]{Koekoeketal}.
\end{proof}
\noindent

\noindent One important transformation for terminating basic hypergeometric series which 
we will use is Watson's $q$-analog of Whipple's theorem which 
relates a terminating balanced ${}_4\phi_3$ to a terminating very-well-poised ${}_8W_7$ (cf.~\cite[\href{http://dlmf.nist.gov/17.9.E15}{(17.9.15)}]{NIST:DLMF}, \cite[Corollary 11]{CohlCostasSantos20b})
\begin{equation}
\label{WatqWhipp}
\qhyp43{q^{-n},a,b,c}{d,e,f}{q,q}=
\frac{\left(\frac{de}{ab},\frac{de}{ac};q\right)}
{\left(\frac{de}{a},\frac{de}{abc};q\right)}
{}_8W_7\left(\frac{de}{qa};q^{-n},\frac{d}{a},\frac{e}{a},b,c;q,\frac{qa}{f}\right),
\end{equation}
where $q^{1-n}abc=def$.
In \cite[Exercise 1.4ii]{GaspRah}, one finds the inversion formula for
terminating basic hypergeometric series.

\begin{thm}[Gasper \& Rahman's (2004) Inversion Theorem] \label{thm:2.2}
Let $m, n, k, r, s\in\mathbb N_0$, $1\le k\le r$, $1\le m\le s$, $a_k\in\mathbb C$, $b_m\not \in\Omega^n_q$,
$q,z\in\mathbb C^\ast$ such that $|q|\ne 1$.
Then,
\begin{eqnarray}
&&\hspace{-3.5cm}\qhyp{r+1}{s}{q^{-n},a_1, \ldots,a_r}{b_1, \ldots,b_s}{q,z}= \frac{(a_1, \ldots,a_r;q)_n}
{(b_1, \ldots,b_s;q)_n}\left(\frac{z}{q}\right)^n\left((-1)^nq^{\binom{n}{2}}\right)^{s-r-1} \nonumber \\
&& \hspace{0.0cm}\label{inversion} \times \sum_{k=0}^n \frac{\left(q^{-n},\frac{q^{1-n}}{b_1}, \ldots,\frac{q^{1-n}}{b_s};q\right)_k}
{\left(q,\frac{q^{1-n}}{a_1}, \ldots,\frac{q^{1-n}}{a_r};q\right)_k} \left(\frac{b_1 \cdots b_s}{a_1 \cdots a_r}
\frac{q^{n+1}}{z}\right)^k.
\end{eqnarray}
\end{thm}
\noindent From the above inversion formula \eqref{inversion},
one may derive the following useful terminating basic hypergeometric transformation
lemma.
\begin{lem}
\label{lem:2.3}
Let $p\in\Z$, $r\in\N_0\cup\{-1\}$, 
$m, n, k, s\in\mathbb N_0$, $1\le k\le r$, $1\le m\le s$, 
$z\in\CCast$, $a_k\in\mathbb C$, $b_m\not \in\Omega^n_q$,
$q\in\mathbb C^\ast$ such that $|q|\ne 1$.
Then, one has
\begin{eqnarray}
&&\hspace{-2.7cm}\qphyp{r+1}{s}{p}{q^{-n},a_1, \ldots,a_r}{b_1, \ldots,b_{s}
}{q,z}= \frac{(a_1, \ldots,a_r; q)_n}{(b_1, \ldots,b_{s};q)_n}
\left(\frac{z}{q}\right)^n\left((-1)^nq^{\binom{n}{2}}\right)^{s-r+p-1}
\nonumber \\ && \hspace{0.0cm}\times \qphyp{s+1}{r}{s-r+p}
{q^{-n},\frac{q^{1-n}}{b_1}, \ldots,\frac{q^{1-n}}{b_{s}}}
{\frac{q^{1-n}}{a_1}, \ldots,\frac{q^{1-n}}{a_r}
}{q,\frac{b_1 \cdots b_{s}}
{a_1 \cdots a_r}\frac{q^{(1-p)n+p+1}}{z}}.
\label{ivg4}
\end{eqnarray}
\end{lem}
\begin{proof}
In a straightforward calculation, if we write \eqref{inversion} and we 
apply \eqref{poch.id:3} assuming all the parameters are nonzero, and then 
we apply the above limit transition identities 
one obtains 
\eqref{ivg4}.
\end{proof}

\begin{cor}\label{cor:2.4}
Let $n,r\in\mathbb N_0$, $q,z\in\mathbb C^\ast$ such that $|q|\ne 1$, and 
for $1\le k\le r$, let $a_k, b_k\not\in\Omega^n_q\cup\{0\}$.
Then,
\begin{eqnarray}
&&\hspace{-1.2cm}\qhyp{r+1}{r}{q^{-n},a_1, \ldots,a_r}{b_1, \ldots,b_r}{q,z}\nonumber\\
&&=
\label{cor:2.4:r1}
q^{-\binom{n}{2}}
(-1)^n
\frac{(a_1, \ldots,a_r;q)_n}
{(b_1, \ldots,b_r;q)_n}
\qhyp{r+1}{r}{q^{-n},
\frac{q^{1-n}}{b_1}, \ldots,
\frac{q^{1-n}}{b_r}}
{\frac{q^{1-n}}{a_1}, \ldots,
\frac{q^{1-n}}{a_r}}{q,
\frac{q^{n+1}}{z}\frac{b_1\cdots b_r}
{a_1\cdots a_r}}.
\end{eqnarray}
\end{cor}
\begin{proof}
Taking $r=s$, $p=0$ in \eqref{ivg4} completes
the proof.
\end{proof}
\noindent Note that in Corollary \ref{cor:2.4}
if the terminating basic hypergeometric
series on the left-hand side is balanced
then the argument of the terminating
basic hypergeometric series on the right-hand side is $q^2/z$.
Applying Corollary \ref{lem:2.3} to the definition
of ${}_{r+1}W_r$, we obtain the following result for a terminating very-well-poised basic hypergeometric series ${}_{r+1}W_r$.%}
\begin{cor} \label{cor:2.5}
Let 
$n, r\in\mathbb N_0$, $a,a_k,q,z\in\CCast$, $\sqrt{a},q^{n+1}a,\frac{q a}{a_k},%\frac{q^{-n}}{\sqrt{a}},
\frac{q^{1-n}}{a},\frac{q^{1-n}}{a_k}\not\in\Omega_q^n$, $k=5, \ldots,r+1$. 
Then, one has the following transformation formula for a very-well-poised terminating 
basic hypergeometric series:
\begin{eqnarray}
&&\hspace{-0.2cm}{}_{r+1}W_r\left(a;q^{-n},a_5, \ldots,a_{r+1};q,z\right)=
q^{-\binom{n}{2}}\left(\frac{-z}{q}\right)^n
\frac{(\pm q\sqrt{a},a,a_5, \ldots,a_{r+1};q)_n}
{\left(\pm \sqrt{a},q^{n+1}a,\frac{qa}{a_5}, \ldots,\frac{qa}{a_{r+1}};q\right)_n}
\nonumber\\
&&\hspace{5cm}\times{}_{r+1}W_r\left(\frac{q^{-2n}}{a};q^{-n},
\frac{q^{-n}a_5}{a}, \ldots,\frac{q^{-n}a_{r+1}}{a};q,\frac{q^{2n+r-3}a^{r-3}}
{(a_5\cdots a_{r+1})^2 z}\right).\nonumber
\end{eqnarray}
\end{cor}
\begin{proof}
Use Corollary \ref{lem:2.3} and \eqref{eq:2.13}.
\end{proof}

\noindent 
Some interesting and useful consequences of this formula are the $r=3$ special case,
\begin{equation}
\label{4W3inverse}
{}_4W_3\left(a;q^{-n};q,z\right)=q^{-\binom{n}{2}}\left(\frac{-z}{q}\right)^n
\frac{(\pm q\sqrt{a},a;q)_n}
{(\pm\sqrt{a},q^{n+1}a;q)_n}
{}_4W_3
\left(\frac{q^{-2n}}{a};q^{-n};q,\frac{q^{2n}}{z}\right),
\end{equation}
as well as the $r=7$ specialization 
\begin{eqnarray}
&&\hspace{-1.2cm}{}_8W_7\left(a;q^{-n},c,d,e,f;q,z\right)
=q^{-\binom{n}{2}} \left(\frac{-z}{q}\right)^n
\frac{\left(\pm q\sqrt{a},a,c,d,e,f;q\right)_n}
{\left(\pm\sqrt{a},q^{n+1}a,\frac{qa}{c},\frac{qa}{d},\frac{qa}{e},\frac{qa}{f};q\right)_n}\nonumber\\
&&\hspace{3.5cm}\times
{}_8W_7\left(\frac{q^{-2n}}{a};q^{-n},
\frac{q^{-n}c}{a},
\frac{q^{-n}d}{a},
\frac{q^{-n}e}{a},
\frac{q^{-n}f}{a};
q,\frac{q^{2n+4}a^4}{z(cdef)^2}
\right).
\label{inv87}
\end{eqnarray}
\noindent Note that 
%in the case when the 
if 
one
obtains an ${}_8W_7$ from a balanced 
${}_4\phi_3$ using \eqref{WatqWhipp},
then 
$q^{2n+4}a^4/(z(cdef)^2)=z.$
Another equality we can use is the %following 
next 
connecting relation between 
basic hypergeometric series on $q$, and on $q^{-1}$:
\begin{eqnarray} 
&&\hspace{-0.5cm}\qhyp{r+1}{r}{q^{-n},a_1, \ldots,a_r}{b_1, \ldots,b_r}{q,z}=
\qhyp{r+1}{r}{q^{n}, a^{-1}_1, \ldots, a^{-1}_r}
{b^{-1}_1, \ldots, b^{-1}_r}{q^{-1}, 
\dfrac{a_1 a_2\cdots a_r}
{b_1 b_2\cdots b_r}\dfrac{z}{q^{n+1}}}\nonumber\\
&&\hspace{0.5cm}=
q^{-\binom{n}{2}}
\left(-\frac zq\right)^n
\frac{(a_1, \ldots,a_r;q)_n}
{(b_1, \ldots,b_r;q)_n}
\qhyp{r+1}{r}{q^{-n},
\frac{q^{1-n}}{b_1}, \ldots,
\frac{q^{1-n}}{b_r}}
{\frac{q^{1-n}}{a_1}, \ldots,
\frac{q^{1-n}}{a_r}}
{q,\frac{b_1\cdots b_r}{a_1\cdots a_r}\frac{q^{n+1}}{z}}.
\label{qtopiden} 
\end{eqnarray}
We will obtain new transformations for basic hypergeometric orthogonal
polynomials by taking advantage of the following remark.

\begin{rem} \label{rem:2.8}
Observe in the following discussion we will often be referring
to a collection of constants $a,b,c,d,e,f$. In such cases, which will be
clear from context, then the constant $e$ should not be confused
with Euler's number, the base of the natural logarithm, i.e.,
$\log\expe=1$.
\end{rem}

In Section \ref{Introduction} we describe our chosen mathematical notations in connection with the terminating basic hypergeometric functions which arise in the symmetric $q$ and $q^{-1}$-Askey-scheme. In Section \ref{askeywilson} we describe the Askey--Wilson polynomials, their terminating representations and transformation properties. 
In Sections \ref{cdqH}, \ref{alsalamchihara}, \ref{cbqH} and \ref{cqH} we provide the terminating representations, terminating transformation properties and summation formulas for the $q$ and $q^{-1}$-symmetric subfamilies of the Askey--Wilson polynomials.
In Section \ref{symmetrygroup}, we conclude by describing the symmetric group structure for the transformations which we have derived.

%%%%%%%%%%%%%%%%%%%%%%%%%%%%%%%%%%%%%%%%%%%%%
\section{The Askey--Wilson polynomials and their symmetric subfamilies}
\label{askeywilson}
%%%%%%%%%%%%%%%%%%%%%%%%%%%%%%%%%%%%%%%%%%%%%

Define the set ${\bf 4}:=\{1,2,3,4\}$, and the multiset 
${\bf a}:=\{a_1,a_2,a_3,a_4\}$, $a_k\in\CCast$, $k\in{\bf 4}$, and
$x=\frac12(z+z^{-1})$.
The Askey--Wilson polynomials $p_n(x;{\bf a}|q)$ are a family of basic hypergeometric orthogonal polynomials 
symmetric in four free parameters. 
These polynomials have a long and in-depth history and their properties 
have been studied in detail. See for instance Ismail (2009) \cite{Ismail:2009:CQO} and references therein.
The basic hypergeometric series representation of the Askey--Wilson polynomials 
fall into four main categories:~(1) terminating ${}_4\phi_3$ representations; 
(2) terminating ${}_8W_7$ representations; (3) nonterminating ${}_8W_7$ 
representations; and (4) nonterminating ${}_4\phi_3$ representations.

The Askey--Wilson polynomials are symmetric with respect to its 
four free parameters, that is, they remain unchanged upon interchange of any two of the
four free parameters. 
Using the properties of $q$-Pochhammer symbols, it is straightforward to replace 
$q\mapsto 1/q$ in the complex plane in order to obtain an extension of these 
polynomials with $|q|>1$. One often refers to polynomials obtained
as such as $q^{-1}$ or $1/q$ polynomials. Since these algebraic factors 
are difficult to search on in the literature, we refer to this extension 
specifically as the $q^{-1}$-Askey--Wilson polynomials. 

First one may obtain the continuous dual $q$-Hahn polynomials by taking the 
limit as the fourth parameter in the Askey--Wilson polynomials tends to zero. 
The continuous dual $q$-Hahn polynomials are symmetric 
in three parameters, as are their $q^{-1}$ analogue, the continuous dual 
$q^{-1}$-Hahn polynomials. 
We derive hypergeometric representations for these polynomials, as well as 
for the Al-Salam--Chihara polynomials which are obtained by taking the 
limit as the third parameter tends to zero in the continuous dual $q$-Hahn polynomials. 
The Al-Salam--Chihara polynomials are symmetric in two parameters, 
as are the resulting $q^{-1}$-Al-Salam--Chihara polynomials. 
By taking the limit in the Al-Salam--Chihara polynomials as one of the two 
free parameters tends to zero, one obtains the continuous big $q$-Hermite polynomials. 
Taking the the limit of the sole remaining free parameter to zero in the continuous 
big $q$-Hermite polynomials, yields the continuous $q$-Hermite polynomials. 
We derive terminating basic hypergeometric representations for all of these 
polynomials as well as their $q^{-1}$ analogues. 
For lack of a better term, we refer to all these families as the {\it symmetric} 
subfamilies of the Askey--Wilson polynomials.

It should however be noted that while the Askey--Wilson polynomials represent 
an infinite-family of orthogonal polynomials, orthogonal with respect to a 
weight function on $[-1,1]$ \cite{MR1164075}, \cite[(14.1.2)]{Koekoeketal} (which gives restrictions 
on the values of the free parameters), the $q^{-1}$-Askey--Wilson polynomials 
represent a finite-family of basic hypergeometric orthogonal polynomials 
\cite{IsmailZhang2022}. 
The symmetric subfamilies of the Askey--Wilson and 
the $q^{-1}$-Askey--Wilson polynomials represent infinite families of 
basic hypergeometric orthogonal polynomials (see also \cite{IsmailZhang2022}). 
In this paper, we will not treat any of the properties of the 
symmetric subfamilies of the Askey--Wilson polynomials related to their 
orthogonality, and therefore we will not require corresponding restrictions on the values of the 
free-parameters which are induced by their orthogonality properties.

In the sequel, from the terminating basic hypergeometric representations 
of the Askey--Wilson polynomials, we derive terminating basic hypergeometric 
representations for the $q^{-1}$-Askey--Wilson polynomials.
From these two families of basic hypergeometric polynomials, one can easily 
derive, using known properties of basic hypergeometric functions, basic 
hypergeometric representations of the symmetric subfamilies of the Askey--Wilson 
and the $q^{-1}$-Askey--Wilson polynomials. 

Different series representations are useful for obtaining different 
properties and formulas for these polynomials. So it is very useful 
to have at hand an exhaustive list.
The discussion contained in the 
remainder of this section is an attempt to summarize, 
in an in-depth manner, an exhaustive description of the representation 
and transformation properties of the {\it terminating}
${}_4\phi_3$ and ${}_8W_7$ basic hypergeometric representations of 
the Askey--Wilson polynomials. 
There is a directed graph to each and every series representation 
which is useful to understand the origin and behavior of these representations.
The remaining sections below which contain the terminating basic hypergeometric 
representations and transformations for the {\it symmetric subfamilies} of the 
Askey--Wilson polynomials all follow from the results presented in this section.

%\newpage
%%%%%%%%%%%%%%%%%%%%%%%%%%%%%%%%%%%%%%%%%%%%%
%\subsection{Terminating Askey--Wilson polynomial representations}
%%%%%%%%%%%%%%%%%%%%%%%%%%%%%%%%%%%%%%%%%%%%%

One has the complete sets of balanced ${}_4\phi_3$ and very-well-poised ${}_8W_7$ terminating basic hypergeometric representations for the Askey--Wilson polynomials given in \cite[Theorem 7]{CohlCostasSantos20b}. The reader should refer there, as we will not reproduce those results here. Note that in order to compare \cite[Theorem 7]{CohlCostasSantos20b} with our results one must replace $\expe^{i\theta}\mapsto z\in\CCast$.
We should also mention that in \cite[Theorem 7]{CohlCostasSantos20b}, applying \eqref{inversion} to \cite[(15), (16), (17), (19)]{CohlCostasSantos20b},
 takes 
$z\mapsto z^{-1}$, and applying it to \cite[(18)]{CohlCostasSantos20b} interchanges $a_p$ and $a_r$. Similar mapping properties certainly apply below, but we will not mention them further.
In \cite[Corollaries 8, 9]{CohlCostasSantos20b}, one has a collection of transformation formulas for the terminating very-well-poised ${}_8W_7$ and terminating balanced ${}_4\phi_3$ which arise from the terminating basic hypergeometric representations of the Askey--Wilson polynomials.

%Let $p_n(x;{\bf a}|q)$, ${\bf a}:=\{a_1,a_2,a_3,a_4\}$, $a_k\in\CCast$, $k\in{\bf 4}$, $q\in\CCdag$. 
One may also consider the
$q^{-1}$-Askey--Wilson polynomials
$p_n(x;{\bf a}|q^{-1})$. Note however that they 
are simply given by
\begin{eqnarray*}
&&\hspace{-1cm}
p_n(x;{\bf a}|q^{-1})
=q^{-3\binom{n}{2}}(-a_{1234})^np_n(x;{\bf a}^{-1}|q),
\end{eqnarray*}
which are just rescaled Askey--Wilson polynomials. This follows by setting $q\mapsto q^{-1}$ in \cite[Theorem 7]{CohlCostasSantos20b} and using \eqref{qtopiden}.
For further information on the terminating transformations
associated with the Askey--Wilson polynomials which includes
all 4-parameter symmetric terminating 
interchange transformations, see \cite{CohlCostasSantos20b}.

\begin{rem} \label{rem:2.7}
Since $x=\frac12(z+z^{-1})$ is 
invariant under the 
map $z\mapsto z^{-1}$, 
all functions in $x$ will be as well.
\end{rem}

%%%%%%%%%%%%%%%%%%%%%%%%%%%%%%%%%%%%%%%%%%%%%
\section{The continuous dual \textit{q} and \textit{q}\textsuperscript{-1}
\!\!-Hahn polynomials}
\label{cdqH}
%%%%%%%%%%%%%%%%%%%%%%%%%%%%%%%%%%%%%%%%%%%%%
Define the set ${\bf 3}:=\{1,2,3\}$, and the multiset 
${\bf a}:=\{a_1,a_2,a_3\}$, $a_k\in\CCast$, 
$k\in{\bf 3}$, and $x=\frac12(z+z^{-1})$, $q\in\CCdag$.
The continuous dual $q$-Hahn polynomials $p_n(x;{\bf a}|q)$ are a family of polynomials 
symmetric in three free parameters. These can be obtained from the 
Askey--Wilson polynomials by taking one of its free parameters to be zero.
Hence, the continuous dual $q$-Hahn polynomials are a symmetric subfamily of the 
Askey--Wilson polynomials. 
%%%%%%%%%%%%%%%%%%%%%%%%%%%%%%%%%%%%%%%%%%%%%
\subsection{Terminating continuous dual \textit{q} and $q^{-1}$-Hahn polynomial representations}
%%%%%%%%%%%%%%%%%%%%%%%%%%%%%%%%%%%%%%%%%%%%%

\noindent One has the following complete list of terminating ${}_3\phi_2$ and ${}_7W_6^{\pm 1}$ basic hypergeometric representations for the continuous dual $q$-Hahn polynomials. 

%\begin{thm} 
\begin{cor} \label{cor:4.1} 
Let $n\in\mathbb N_0$, 
$q\in\CCdag$, 
$p,s,r,t\in{\bf 3}$, $p,r,t$ distinct and fixed.
Then, the continuous dual $q$-Hahn polynomials are given by:
\begin{eqnarray}
\label{cdqH:def1} &&\hspace{-2.12cm}p_n(x;{\bf a}|q) := a_p^{-n} (\{a_{ps}\}_{s \neq p};q)_n 
\qhyp{3}{2}{q^{-n}, a_pz^{\pm}}{\{a_{ps}\}_{s \neq p}}{q,q} \\
\label{cdqH:def2} &&\hspace{-0.35cm}=q^{-\binom{n}{2}}(-a_p)^{-n} 
\left(a_pz^{\pm};q\right)_n \qhyp{3}{2}{q^{-n}, 
\{\frac{q^{1-n}}{a_{ps}}\}_{s \neq p}}{\frac{q^{1-n}z^{\pm}}{a_p}}
{q,\frac{q^n a_{123}}{a_p}} \\ \label{cdqH:def3} &&\hspace{-0.35cm}=z^n
\left(a_{pr},\frac{a_t}{z};q\right)_n \qhyp{3}{2}{q^{-n}, a_pz, a_rz}
{a_{pr},\frac{q^{1-n}z}{a_t}}{q,\frac{q}{a_tz}} \\
\label{cdqH:def4} &&\hspace{-0.35cm}=z^n \left(\frac{a_p}{z},\frac{a_r}{z};q\right)_n 
\qhyp{3}{2}{q^{-n},a_tz,\frac{q^{1-n}}{a_{pr}}}
{\frac{q^{1-n}z}{a_p},\frac{q^{1-n}z}{a_r}}{q,q} \\
\label{cdqH:def6}&&\hspace{-0.35cm}=a_p^{-n}\dfrac{\left(a_{pt},a_rz^{\pm};q\right)_n}
{\left(\frac{a_r}{a_p};q\right)_n} 
\Wphyp{7}{6}{1}{\frac{q^{-n}a_p}{a_r}}
{q^{-n},\frac{q^{1-n}}{a_{rt}}, a_pz^{\pm}}
{q,\frac{q a_t}{a_r}}\\
&&
\hspace{-0.35cm}=z^n
\dfrac{(\left\{\frac{a_s}{z}\right\},\!\frac{a_{123}}{qz};q)_{n}}
{\left(\frac{a_{123}}{qz};q\right)_{2n}}
%}\nonumber\\&&
\label{cdqH:def7}
\Wphyp{7}{6}{-1}{\frac{q^{1-2n}z}{a_{123}}}
{q^{-n},\left\{\frac{q^{1-n}a_{123}}{a_s}\right\}}
{q,q^{2n-1}a_{123}z}
\\
&&\label{cdqH:def8}
\hspace{-0.35cm}=z^n\dfrac{(\{\frac{a_{123}}{a_{s}}\};q)_n}
{\left(a_{123}\,z;q\right)_{n}}
\Wphyp{7}{6}{-1}{\frac{a_{123}z}{q}}
{q^{-n},\{a_s z\}}
{q,\frac{q^{n}}{z^2}}
\\
\label{cdqH:def5} &&\hspace{-0.35cm}=z^n \frac{\left(\{\frac{a_s}{z}\};q\right)_n}{(z^{-2};q)_n} 
\Wphyp{7}{6}{-1}{q^{-n}z^2}{q^{-n},\{a_sz\}}
{q,\frac{q}{a_{123}z}}
.
\end{eqnarray}
\end{cor}

\begin{proof}
\eqref{cdqH:def1} is derived by taking \cite[(13)]{CohlCostasSantos20b} 
%\eqref{aw:def1} 
and 
mapping $a_k \mapsto 0$, $k\ne p$
(see also \cite[(14.3.1)]{Koekoeketal}),
\eqref{cdqH:def2} is derived by taking \cite[(14)]{CohlCostasSantos20b} 
%\eqref{aw:def2} 
and taking the limit $a_k \to 0$, $k\ne p$,
\eqref{cdqH:def3} is derived by taking \cite[(15)]{CohlCostasSantos20b} and taking the limit $a_u \to 0$, 
\eqref{cdqH:def4} is derived by taking \cite[(15)]{CohlCostasSantos20b} and taking the limit $a_r \to 0$, 
\eqref{cdqH:def6} is derived by taking \cite[(18)]{CohlCostasSantos20b} and taking the limit $a_u \to 0$, 
\eqref{cdqH:def7}, \eqref{cdqH:def8} are derived by taking \cite[(16)]{CohlCostasSantos20b}, \cite[(17)]{CohlCostasSantos20b} 
respectively and taking the limit $a_p \to 0$,
and 
\eqref{cdqH:def5} is derived by taking \cite[(19)]{CohlCostasSantos20b} and taking the limit $a_4 \to 0$, 
$k\in{\bf 4}$.
Also, whenever necessary the indices $k$ in $a_k$ should be relabeled such that $k\in{\bf 3}$.
\end{proof}

\begin{rem}
\label{invcdqirem}
Let 
${\bf a}=\{a_1,a_2,a_3\}$, 
${\bf b}=\{b_1,b_2,b_3,b_4\}$, 
$q,b_k\in\mathbb C^\ast$, 
$p,k\in{\bf 4}$, $a_m\in\mathbb C^\ast$, 
$m\in\bf 3$, such that
$|q|\ne 1$.
Then, the continuous dual 
$q^{-1}$-Hahn polynomials with 
reciprocal parameters are given by
\begin{equation}
%&&\hspace{-1cm} 
p_n\left(x;{\bf a}^{-1}
|q^{-1}\right)
=
p_n\left(x;{\bf b}_{[p]}^{-1}
|q^{-1}\right)
=q^{-3\binom{n}{2}}\left(\frac{-b_p}{b_{1234}}\right)^n
\lim_{b_p\to\infty}\frac{1}{b_p^n}p_n(x;
{\bf b}
|q),
\end{equation}
%where $p\in{\bf 4}$, 
with
the $p_n$ on the right-hand
side being the Askey--Wilson polynomials.
Whenever necessary the resulting indices
$k\in{\bf 4}\setminus \{p\}$ should be relabeled such that $k\in{\bf 3}$.
So aside from a specific normalization, as is well-known,
the continuous dual $q^{-1}$-Hahn polynomials are the renormalized limit of the Askey--Wilson polynomials as one of the parameters goes
to infinity.
\end{rem}

%%%%%%%%%%%%%%%%%%%%%%%%%%%%%%%%%%%%%%%%%%%%%
%\subsection{Terminating continuous dual $q^{-1}$-Hahn polynomial representations}
%%%%%%%%%%%%%%%%%%%%%%%%%%%%%%%%%%%%%%%%%%%%%
\noindent One has the following complete list of terminating ${}_3\phi_2$ and ${}_7W_6^{\pm 1}$ basic hypergeometric representations for the continuous dual $q^{-1}$-Hahn polynomials. 

\begin{cor}
\label{cor:4.11}
Let $p_n(x;{\bf a}|q)$ and all the respective parameters be defined as previously. Then, the continuous dual 
$q^{-1}$-Hahn polynomials are given by:
\begin{eqnarray}
&&\hspace{-0.9cm}\label{cdqiH:1} p_n(x;{\bf a}|q^{-1}) =
q^{-2\binom{n}{2}}
a_{123}^n 
\left(\left\{\frac{1}{a_{ps}}\right\}_{s \neq p};q\right)_n \qhyp{3}{2}{q^{-n},\frac{z^{\pm}}{a_p}}
{\{\frac{1}{a_{ps}}\}_{s \neq p}}{q,\frac{q^na_p}{a_{123}}} \\
\label{cdqiH:2} &&\hspace{1.3cm}\hspace{-1mm}=q^{-\binom{n}{2}}(-a_p)^n \left(\frac{z^{\pm}}{a_p};q\right)_n 
\qhyp{3}{2}{q^{-n},\{q^{1-n}a_{ps}\}_{s \neq p}}{q^{1-n}a_pz^{\pm}}{q,q} \\
\label{cdqiH:3} &&\hspace{1.3cm}\hspace{-1mm}=q^{-2\binom{n}{2}} a_{123}^n \left(\frac{1}{a_{pr}},\frac{z}
{a_t};q\right)_n 
\qhyp{3}{2}{q^{-n}, 
\frac{1}{a_pz}, 
\frac{1}{a_rz}
}{
\frac{q^{1-n}a_t}{z},
\frac{1}{a_{pr}}
}{q,q} \\
\label{cdqiH:4} &&\hspace{1.3cm} \hspace{-1mm}=
q^{-2\binom{n}{2}} 
\left(\frac{a_{pr}}{z}\right)^n
\left(\frac{z}{a_p},\frac{z}
{a_r};q\right)_n \qhyp{3}{2}{q^{-n},\frac{1}{a_tz},q^{1-n}a_{pr}}
{\frac{q^{1-n}a_p}{z},\frac{q^{1-n}a_r}{z}}
{q,\frac{qa_t}{z}} \\
%\end{eqnarray}
%\begin{eqnarray}
&&\hspace{1.2cm}=
q^{-2\binom{n}{2}} 
a^n_{123}
\frac{\left(
\frac{1}{a_{pt}},\frac{z^{\pm}}{a_r};q\right)_n}
{(\frac{a_p}{a_r};q)_n}%\nonumber \\
\label{cdqiH:6} 
\Wphyp{7}{6}{-1}{\frac{q^{-n}a_r}{a_p}}
{q^{-n}, q^{1-n}a_{rt},\frac{z^{\pm}}{a_p}}
{q,\frac{q^n a_p}{a_{t}}}\\
%%%%%
&&\hspace{1.2cm}=z^{-n}
\dfrac{\left(\left\{\frac{z}{a_s}\right\}_{s\ne p}
,\frac{z}{q a_{123}};q\right)_{n}}
{\left(\frac{z}{q a_{123}};q\right)_{2n}}
%\nonumber\\
%&&\hspace{3cm}\times
\Wphyp{7}{6}{1}{\frac{q^{1-2n}a_{123}}{z}}
{q^{-n},\left\{\frac{q^{1-n}a_{123}}{a_{s}}\right\}}
{q, \frac{q}{z^2}}
\label{cdqiH:7}
\\
&&\label{cdqiH:8}\hspace{1.2cm}
=q^{-2{\binom n2}}%q^{-2{n\choose 2}} 
a_{123}^n \frac{\left(
\left\{\frac{a_{s}}{a_{123}}\right\};q\right)_n}
{\left(\frac{1}{a_{123}z};q\right)_{n}}
\Wphyp{7}{6}{1}{\frac{1}{qa_{123}z}}
{q^{-n},\left\{\frac{1}{a_sz}\right\}}
{q,\frac{q^n\,z}{a_{123}}}\\
%%%%%
&&\hspace{1.2cm}=q^{-2\binom{n}{2}}a^n_{123} 
\frac{\left(\left\{\frac{z}{a_s}\right\};q\right)_n}
{(z^{2};q)_n} %\nonumber \\
\label{cdqiH:5} %&& \hspace{2.7cm} \times 
\Wphyp{7}{6}{1}{\frac{q^{-n}}{z^2}}
{q^{-n},\left\{\frac{1}{a_sz}\right\}}
{q,\frac{q^{2-n}a_{123}}{z}}.
\end{eqnarray}
\end{cor}

\begin{proof}
Each $q^{-1}$ representation is derived from the 
corresponding representation by applying the 
map $q\mapsto 1/q$ and using \eqref{poch.id:3}.
\end{proof}

%\begin{cor} \noro{Let $n\in\mathbb N_0$, $z,q\in\CCast$, 
%such that $|q|\ne 1$.
%Then, one has 
%the following transformation
%formula for a terminating ${}_3\phi_2$:}
%\begin{equation}
%\noro{Insert\ formula.}
%\end{equation}
%\end{cor}
%\begin{proof}
%\noro{Start with \eqref{cdqiH:3}. Insert rest.}
%\end{proof}

%%%%%%%%%%%%%%%%%%%%%%%%%%%%%%%%%%%%%%%%%%%%%
\subsection{Terminating \textit{3}-parameter $q$ and $q^{-1}$ symmetric transformations}
%%%%%%%%%%%%%%%%%%%%%%%%%%%%%%%%%%%%%%%%%%%%%

\noindent By studying the interrelations of the terminating basic hypergeometric functions in Corollary \ref{cor:4.1} we obtain the following transformation result for a terminating ${}_7W_6^{-1}$.

\begin{cor}
\label{cor:4.3}
Let $n\in\mathbb N_0$, $a,c,d,e,\in\CCast$, $q\in\CCdag$.
Then, one has the following transformation
formulas for a terminating 
${}_7W_6^{-1}$:
\begin{eqnarray}
&&\hspace{-1.0cm}\label{cor4.3:r1}
\Wphyp{7}{6}{-1}{a}{q^{-n},c,d,e}{q,\frac{q^{n+1}a}{cde}}
\\
&&\label{cor4.3:r2}\hspace{-0.5cm}=
\left(\frac{qa}{cde}\right)^n\!\!\dfrac{\left(qa,a, c, d, e;q\right)_n}
{(a;q)_{2n}\!\left(\frac{qa}{c},\frac{qa}{d},\frac{qa}{e};q\right)_n} 
\Wphyp{7}{6}{-1}{\frac{q^{-2n}}{a}}
{q^{-n},\frac{q^{-n}c}{a},\frac{q^{-n}d}{a},\frac{q^{-n}e}{a}}
{q,\frac{q^{2n+1}a^2}{cde}}
\\
&&\label{cor4.3:r9}
\hspace{-0.5cm}=
\left(\frac{qa}{cde}\right)^n\!\!\dfrac{\left(qa;q\right)_n}
{\left(\frac{q^2a^2}{cde};q\right)_n} 
\Wphyp{7}{6}{-1}{\frac{qa^2}{cde}}
{q^{-n},\frac{qa}{cd},\frac{qa}{ce},\frac{qa}{de}}
{q,\frac{q^{n-1}cde}{a}}
\\
&&\label{cor4.3:r10}
\hspace{-0.5cm}=
\dfrac{\left(qa,\frac{qa}{cd},\frac{qa}{ce},\frac{qa}{de};q\right)_n}
{\left(\frac{qa}{c},\frac{qa}{d},\frac{qa}{e},\frac{qa}{cde};q\right)_n} 
\Wphyp{7}{6}{-1}{
\frac{q^{-n-1}cde}{a}}
{q^{-n},c,d,e}
{q,\frac{1}{a}}
\\
&&\label{cor4.3:r10b}
\hspace{-0.5cm}=
\left(\frac{qa}{cde}\right)^n\dfrac{(qa,c,d,e;q)_n}
{\left(\frac{qa}{c},\frac{qa}{d},\frac{qa}{e},\frac{cde}{qa};q\right)_n}
\Wphyp{7}{6}{-1}{
\frac{q^{1-n}a}{cde}}
{q^{-n},\frac{qa}{cd},\frac{qa}{ce},\frac{qa}{de}}
{q,\frac{cde}{qa^2}}
\\
&&\label{cor4.3:r10x}
\hspace{-0.5cm}=
\dfrac{(qa,\frac{qa}{cd},\frac{qa}{ce},\frac{qa}{de},\frac{qa^2}{cde};q)_n}
{(\frac{qa^2}{cde};q)_{2n}\left(\frac{qa}{c},\frac{qa}{d},\frac{qa}{e};q\right)_n}
\Wphyp{7}{6}{-1}{
\frac{q^{-2n-1}cde}{a^2}}
{q^{-n},\frac{q^{-n}c}{a},\frac{q^{-n}d}{a},\frac{q^{-n}e}{a}}
{q,q^{2n}a}
\\
&&\label{cor4.3:r3}
\hspace{-0.5cm}=
\left(\frac{1}{c}\right)^n\frac{\left(qa,d,\frac{qa}{ce};q\right)_n}
{\left(\frac{qa}{c},\frac{qa}{e},\frac{d}{c};q\right)_n}
\Wphyp{7}{6}{1}{\frac{q^{-n}c}{d}}
{q^{-n},\frac{q^{-n}c}{a},\frac{qa}{de},c}
{q,\frac{qe}{d}}
\\
&&\label{cor4.3:r4a}\hspace{-0.5cm}=\frac{\left(qa,\frac{qa}{cd};q\right)_n}
{\left(\frac{qa}{c},\frac{qa}{d};q\right)_n}\qhyp{3}{2}{q^{-n},c,d}{\frac{q^{-n}cd}{a},\frac{qa}{e}}{q,\frac{q}{e}}\\
&&\label{cor4.3:r4b}\hspace{-0.5cm}=
q^{-\binom{n}{2}}\left(\frac{-1}{e}\right)^n
\dfrac{\left(qa,e,\frac{qa}{cd}; q\right)_n}
{\left(\frac{qa}{c},\frac{qa}{d},\frac{qa}{e};q\right)_n} 
\qhyp{3}{2}{q^{-n},\frac{q^{-n}c}{a}, 
\frac{q^{-n}d}{a}}{\frac{q^{-n}cd}a,\frac{q^{1-n}}{e}}{q,\frac{q^{n+1}a}{e}}%\\
\end{eqnarray}
\begin{eqnarray}
&&\label{cor4.3:r5}\hspace{-5.7cm}=\left(\frac{qa}{cde}\right)^n\frac{\left(qa,c;q\right)_n}
{\left(\frac{qa}{d},\frac{qa}{e};q\right)_n}
\qhyp{3}{2}{q^{-n},\frac{qa}{cd},\frac{qa}{ce}}
{\frac{q^{1-n}}{c},\frac{qa}{c}}{q,\frac{de}{a}}\\
&&\label{cor4.3:r7}\hspace{-5.7cm}=
\left(\frac{1}{c}\right)^n
\frac{(qa;q)_n}
{\left(\frac{qa}{c};q\right)_n}
\qhyp{3}{2}{q^{-n},\frac{qa}{de},c}
{\frac{qa}{d},\frac{qa}{e}}{q,q}\\
&&\label{cor4.3:r6}\hspace{-5.7cm}=\frac{\left(qa,\frac{qa}{cd},\frac{qa}{ce};q\right)_n}
{\left(\frac{qa}{c},\frac{qa}{d},\frac{qa}{e};q\right)_n}
\qhyp{3}{2}{q^{-n},\frac{q^{-n}c}{a},c}{\frac{q^{-n}cd}{a},\frac{q^{-n}ce}{a}}{q,q}\\
&&\label{cor4.3:r8}\hspace{-5.7cm}
=\left(\frac{qa}{cde}\right)^n
\frac{(qa,d,e;q)_n}
{\left(\frac{qa}{c},
\frac{qa}{d},\frac{qa}{e};q\right)_n}
\qhyp{3}{2}{q^{-n},\frac{q^{-n}c}{a},\frac{qa}{de}}{\frac{q^{1-n}}{d},\frac{q^{1-n}}{e}}{q,q}.
\end{eqnarray}
\end{cor}

\begin{proof}
Start with \cite[Corollaries 8, 9]{CohlCostasSantos20b}, then 
take the limit as $f\to 0$.
Taking the limit as $c\to 0$ of \cite[(23), (24)]{CohlCostasSantos20b},
%\eqref{cor3.5a.7}, \eqref{cor3.5a.7b} 
%in \cite[Corollaries 8, 9]{CohlCostasSantos20b}, 
then replacing $f\mapsto c$ 
by symmetry in $c,d,e,f$ of the ${}_8W_7$ 
produces \eqref{cor4.3:r9} and 
\eqref{cor4.3:r10}.
Taking the limit as $d\to 0$ in \cite[(25), (26)]{CohlCostasSantos20b} and
 then replacing $f\mapsto d$ by 
symmetry in $c,d,e,f$ of the ${}_8W_7$ produces
\eqref{cor4.3:r7} and \eqref{cor4.3:r8}.
This
completes the proof. Similarly, one can start with Corollary \ref{cor:4.1} and set $\{z^2,a_p,a_r,a_t\}\mapsto \{q^na,q^{-\frac{n}{2}}{c}/{\sqrt{a}},q^{-\frac{n}{2}}{d}/{\sqrt{a}},q^{-\frac{n}{2}}/{e}{\sqrt{a}}\}$ and solve for the ${}_7W_6^{-1}$ given in \eqref{cdqH:def5}.
\end{proof}

%%%%%%%%%%%%%%%%%%%%%%%%%%%%%%%%%%%%%%%%%%%%%
%\subsection{Terminating \textit{3}-parameter symmetric $q^{-1}$ transformations}
%%%%%%%%%%%%%%%%%%%%%%%%%%%%%%%%%%%%%%%%%%%%%

\noindent By studying the interrelations of the terminating basic hypergeometric functions in Corollary \ref{cor:4.11} we obtain the following transformation result for a terminating ${}_7W_6^{1}$.

\begin{cor}
\label{cor:4.13}
Let $n\in\mathbb N_0$, $a,c,d,e\in\CCast$,
$q\in\CCdag$.
Then, one has 
the following transformation
formulas for a terminating 
${}_7W_{6}^{1}$:
\begin{eqnarray}
&&\hspace{-0.8cm}
\Wphyp{7}{6}{1}{a}{q^{-n},c,d,e}{q,\frac{q^{n+2}a^2}{cde}}
\label{cor:4.13-0}
\\[-0.05cm]
&&=
q^{2\binom{n}{2}}\left(\frac{q^2a^2}{cde}\right)^n
\frac{(qa,a,c,d,e;q)_n}{(a;q)_{2n}(\frac{qa}{c},\frac{qa}{d},\frac{qa}{e};q)_n}
\Wphyp{7}{6}{1}{\frac{q^{-2n}}{a}}
{q^{-n},\frac{q^{-n}c}{a},\frac{q^{-n}d}{a},\frac{q^{-n}e}{a}}
{q,\frac{q^2a}{cde}}
\label{cor:4.13-1}\\[-0.05cm]
&&\label{cor:4.13-9}=
\frac{(qa,\frac{qa}{cd},\frac{qa}{ce},\frac{qa}{de};q)_n}
{(\frac{qa}{c},\frac{qa}{d},\frac{qa}{e},\frac{qa}{cde};q)_n}
\Wphyp{7}{6}{1}{\frac{q^{-n-1}cde}{a}}
{q^{-n},c,d,e}
{q,\frac{q^{-n}cde}{a^2}}
\\[-0.05cm]
&&\label{cor:4.13-10}=
\frac{(qa;q)_n}{(\frac{q^2a^2}{cde};q)_n}
\Wphyp{7}{6}{1}{\frac{qa^2}{cde}}
{q^{-n},\frac{qa}{cd},\frac{qa}{ce},\frac{qa}{de}}
{q,q^{n+1}a}
\\[-0.05cm]
&&\label{cor:4.13-10b}
=
\frac{(qa,c,d,e;q)_n}{(\frac{qa}{c},\frac{qa}{d},\frac{qa}{e},\frac{cde}{qa};q)_n}
\Wphyp{7}{6}{1}{\frac{q^{1-n}a}{cde}}
{q^{-n},\frac{qa}{cd},\frac{qa}{ce},\frac{qa}{de}}
{q,\frac{q^{1-n}}{a}}
\\[-0.05cm]
&&\label{cor:4.13-10x}
=
\frac{q^{2\binom{n}{2}}(qa)^n(qa,\frac{qa}{cd},\frac{qa}{ce},\frac{qa}{de},\frac{qa^2}{cde};q)_n}{(\frac{qa^2}{cde};q)_{2n}(\frac{qa}{c},\frac{qa}{d},\frac{qa}{e};q)_n}
\Wphyp{7}{6}{1}{\frac{q^{-2n-1}cde}{a^2}}
{q^{-n},\frac{q^{-n}c}{a},\frac{q^{-n}d}{a},\frac{q^{-n}e}{a}}
{q,\frac{cde}{a}}
\\[-0.05cm]
&&=
\frac{(qa,d,\frac{qa}{ce};q)_n}{\left(\frac{qa}{c},\frac{qa}{e},\frac{d}{c};q\right)_n}
\Wphyp{7}{6}{-1}{\frac{q^{-n}c}{d}}
{q^{-n},\frac{q^{-n}c}{a},\frac{qa}{de},c}
{q,\frac{q^ne}{c}}
\label{cor:4.13-2}\\[-0.05cm]
&&=
\frac{\left(qa;q\right)_n}
{(\frac{qa}{e};q)_n}
\qhyp{3}{2}{q^{-n},\frac{qa}{cd},e}{\frac{qa}{c},\frac{qa}{d}}{q,\frac{q^{n+1}a}{e}}
\label{cor:4.13-3}\\[-0.05cm]
&&=\left(\frac{qa}{cd}\right)^n\frac{\left(qa,c,d;q\right)_n}
{(\frac{qa}{c},\frac{qa}{d},\frac{qa}{e};q)_n}
\qhyp{3}{2}{q^{-n},\frac{q^{-n}e}{a},\frac{qa}{cd}}{\frac{q^{1-n}}{c},\frac{q^{1-n}}{d}}{q,
\frac{q}{e}}\label{cor:4.13-4}\\[-0.05cm]
&&=c^n\frac{\left(qa,\frac{qa}{cd},\frac{qa}{ce};q\right)_n}
{(\frac{qa}{c},\frac{qa}{d},\frac{qa}{e};q)_n}
\qhyp{3}{2}{q^{-n},\frac{q^{-n}c}{a},c}{\frac{q^{-n}cd}{a},\frac{q^{-n}ce}{a}}{q,\frac{de}{a}}.
\label{cor:4.13-8}\\[-0.05cm]
&&=
\frac{\left(qa,\frac{qa}{cd};q\right)_n}{(\frac{qa}{c},\frac{qa}{d};q)_n}
\qhyp{3}{2}{q^{-n},c,d}{\frac{q^{-n}cd}{a},\frac{qa}{e}}{q,q}
\label{cor:4.13-5}
\\[-0.05cm]
&&=q^{\binom{n}{2}}
\left(\frac{-{qa}}{e}\right)^n\frac{\left(qa,e,\frac{qa}{cd};q\right)_n}
{(\frac{qa}{c},\frac{qa}{d},\frac{qa}{e};q)_n}
\qhyp{3}{2}{q^{-n},\frac{q^{-n}c}{a},\frac{q^{-n}d}{a}}{\frac{q^{-n}cd}{a},\frac{q^{1-n}}{e}}{q,q}\label{cor:4.13-6}\\[-0.05cm]
&&=
\frac{\left(qa,c;q\right)_n}
{\left(\frac{qa}{d},\frac{qa}{e};q\right)_n}
\qhyp{3}{2}{q^{-n},\frac{qa}{cd},\frac{qa}{ce}}{\frac{qa}{c},\frac{q^{1-n}}{c}}{q,q}\label{cor:4.13-7}.
\end{eqnarray}
\end{cor}
\begin{proof}
Start with Corollary \ref{cor:4.11} and set $z^{2}
=q^na$, $a_p=q^{\frac{n}{2}}\frac{\sqrt{a}}{c}$,
$a_r=q^{\frac{n}{2}}\frac{\sqrt{a}}{d}$,
$a_t=q^{\frac{n}{2}}\frac{\sqrt{a}}{e}$.
Then, multiply every formula by the factor
\[
{\sf A}_n(a,c,d,e|q):=
\frac{q^{\binom{n}{2}}(qa)^\frac{n}{2}}{c^n(\frac{qa}{cd},\frac{qa}{ce};q)_n}.
\]
With simplification, this completes the proof.
\end{proof}

%%%%%%%%%%%%%%%%%%%%%%%%%%%%%%%%%%%%%%%%%%%%%
\subsection{Terminating \textit{3}-parameter $q$ and $q^{-1}$ symmetric interchange transformations}
%%%%%%%%%%%%%%%%%%%%%%%%%%%%%%%%%%%%%%%%%%%%%
\begin{cor} \label{cor:4.4} Let $n\in\mathbb N_0$, $a,c,d,e\in\CCast$,
$q\in\CCdag$.
Then, one has the following parameter interchange 
transformations for a terminating 
${}_{7}W_{6}^{1}$:
\begin{eqnarray}
&&\hspace{-2cm}\label{cor4.4:r1}
\hspace{0.5cm}
\Wphyp{7}{6}{1}{\frac{q^{-n}c}{d}}
{q^{-n},\frac{q^{-n}c}{a},\frac{qa}{de},c}
{q,\frac{qe}{d}}
\\[-0.05cm]
&&\label{cor4.4:r2}
\hspace{0.5cm}=
\frac{
(\frac{qa}{cd},\frac{qa}{e},\frac{d}{c};q)_n
}{
(\frac{qa}{ce},\frac{qa}{d},\frac{e}{c};q)_n
}
\Wphyp{7}{6}{1}{\frac{q^{-n}c}{e}}
{q^{-n},\frac{q^{-n}c}{a},\frac{qa}{ed},c}
{q,\frac{qd}{e}}
\\[-0.05cm]
&&\label{cor4.4:r3}
\hspace{0.5cm}=
\left(\frac{c}{d}\right)^n
\frac{
(\frac{qa}{de},\frac{qa}{c},\frac{d}{c},c;q)_n 
}{
(\frac{qa}{ce},\frac{qa}{d},\frac{c}{d},d;q)_n
}
\Wphyp{7}{6}{1}{\frac{q^{-n}d}{c}}
{q^{-n},\frac{q^{-n}d}{a},\frac{qa}{ce},d}
{q,\frac{qe}{c}}\\[-0.05cm]
&&\label{cor4.4:r4}
\hspace{0.5cm}=
\left(\frac{c}{d}\right)^n
\frac{
(\frac{qa}{dc},\frac{qa}{e},\frac{d}{c},e;q)_n
}{
(\frac{qa}{ce},\frac{qa}{d},\frac{e}{d},d;q)_n
}
\Wphyp{7}{6}{1}{\frac{q^{-n}d}{e}}
{q^{-n},\frac{q^{-n}d}{a},\frac{qa}{ec},d}
{q,\frac{qc}{e}}
\\[-0.05cm]
&&\label{cor4.4:r5}
\hspace{0.5cm}=
\left(\frac{c}{e}\right)^n
\frac{
(\frac{qa}{ed},\frac{qa}{c},\frac{d}{c},c;q)_n
}{
(\frac{qa}{ce},\frac{qa}{d},\frac{c}{e},d;q)_n
}
\Wphyp{7}{6}{1}{\frac{q^{-n}e}{c}}
{q^{-n},\frac{q^{-n}e}{a},\frac{qa}{cd},e}
{q,\frac{qd}{c}}
\\[-0.05cm]
&&\label{cor4.4:r6}
\hspace{0.5cm}=
\left(\frac{c}{e}\right)^n
\frac{
(\frac{d}{c};q)_n
}{
(\frac{d}{e};q)_n
}
\Wphyp{7}{6}{1}{\frac{q^{-n}e}{d}}
{q^{-n},\frac{q^{-n}e}{a},\frac{qa}{dc},e}
{q,\frac{qc}{d}}.
\end{eqnarray}
\end{cor}
\begin{proof}
\noindent
Start with \eqref{cor4.3:r3} and consider all permutations of the symmetric parameters $c,d,e$ which produce non-trivial transformations.
\end{proof}

\begin{cor}\label{cor:4.5} Let $n\in\mathbb N_0$, $a,c,d,e\in\CCast$, 
$q\in\CCdag$.
%such that $|q|\ne 1$. 
Then, one has the following parameter interchange transformations for a terminating ${}_3\phi_2$:
\begin{eqnarray}
&&\hspace{-2cm}\label{cor4.5:r1}\hspace{0.5cm}
\qhyp{3}{2}{q^{-n},c,d}{\frac{q^{-n}cd}{a},\frac{qa}{e}}{q,\frac{q}{e}}
\\[-0.05cm]&&\label{cor4.3:r4a:r1}\hspace{0.5cm}
\!=\!
\frac{(\frac{qa}{ce},\frac{qa}{d};q)_n}{
(\frac{qa}{cd},\frac{qa}{e};q)_n}
\qhyp{3}{2}{q^{-n},c,e}{\frac{q^{-n}ce}{a},\frac{qa}{d}}{q,\frac{q}{d}}
\\[-0.05cm]&&\label{cor4.5:r4a:r1}\hspace{0.5cm}
\!=\!
\frac{(\frac{qa}{de},\frac{qa}{c};q)_n}{
(\frac{qa}{cd},\frac{qa}{e};q)_n}
\qhyp{3}{2}{q^{-n},d,e}{\frac{q^{-n}de}{a},\frac{qa}{c}}{q,\frac{q}{c}}.
\end{eqnarray}
\end{cor}
\begin{proof}
\noindent 
Start with \eqref{cor4.3:r4a} and consider all permutations of the symmetric parameters $c,d,e$ which produce non-trivial transformations.
\end{proof}

\begin{rem}
Another set of parameter interchange transformations
can be obtained by considering all permutations
of the symmetric parameters $c,d,e$ in
\eqref{cor4.3:r4b} and \eqref{cor4.3:r5}
respectively. However, one can see that
these are equivalent to the above Corollary 
\ref{cor:4.5} by replacing
\begin{eqnarray}
&&\hspace{-5.0cm}(a,c,d,e)\mapsto \left(\frac{q^{-2n}}a,\frac{q^{-n}c}a,\frac{q^{-n}d}a,\frac{q^{-n}e}a\right),\\
&&\hspace{-5.0cm}(a,c,d,e)\mapsto \left(\frac{qa}{ce},\frac{qa}{de},\frac{qa}{cd},\frac{qa}{cde}\right),
\end{eqnarray}
respectively.
\end{rem}

\begin{cor}\label{cor:4.9} Let $n\in\mathbb N_0$, $a,c,d,e\in\CCast$, 
$q\in\CCdag$.
%such that $|q|\ne 1$. 
Then, one has the following parameter 
interchange transformations for a terminating ${}_3\phi_2$:
\begin{eqnarray}
&&\label{cor4.9:r1}\hspace{0.5cm}
\hspace{-2cm}\qhyp{3}{2}{q^{-n},\frac{qa}{de},c}
{\frac{qa}{d},\frac{qa}{e}}{q,q}
%\\[-0.05cm]&&\label{cor4.3:r7:r1}\hspace{0.5cm}
=
\left(\frac{c}{d}\right)^n
\frac{
(\frac{qa}{c};q)_n
}
{(\frac{qa}{d};q)_n}
\qhyp{3}{2}{q^{-n},\frac{qa}{ce},d}
{\frac{qa}{c},\frac{qa}{e}}{q,q}
\\[-0.05cm]&&\label{cor4.9:r2}\hspace{2.2cm}
=
\left(\frac{c}{e}\right)^n
\frac{
(\frac{qa}{c};q)_n
}
{(\frac{qa}{e};q)_n}
\qhyp{3}{2}{q^{-n},\frac{qa}{cd},e}
{\frac{qa}{c},\frac{qa}{d}}{q,q}\!.
\end{eqnarray}
\end{cor}
\begin{proof}
\noindent
Start with \eqref{cor4.3:r7} and consider all permutations of the symmetric parameters $c,d,e$ which produce non-trivial transformations.
\end{proof}

\begin{rem}
Another set of parameter interchange transformations
can be obtained by considering all permutations
of the symmetric parameters $c,d,e$ in
\eqref{cor4.3:r6} and \eqref{cor4.3:r8}
respectively. However, one can see that
these are equivalent to the above Corollary 
\ref{cor:4.9} by replacing
\begin{eqnarray}
&&\hspace{-5.0cm}(a,c,d,e)\mapsto \left(\frac{q^{-1-n}cde}{a}, c,d,e\right),\\
&&\hspace{-5.0cm}(a,c,d,e)\mapsto \left(\frac{q^{1-n}a}{cde},\frac{qa}{de},\frac{qa}{ec},\frac{qa}{cd}\right),
\end{eqnarray}
respectively.
\end{rem}

%%%%%%%%%%%%%%%%%%%%%%%%%%%%%%%%%%%%%%%%%%%%%
%\subsection{Terminating \textit{3}-parameter symmetric 
%$q^{-1}$ interchange transformations}
%%%%%%%%%%%%%%%%%%%%%%%%%%%%%%%%%%%%%%%%%%%%%
\begin{cor} Let $n\in\mathbb N_0$, $a,c,d,e\in\CCast$, 
$q\in\CCdag$.
%such that $|q|\ne 1$. 
Then, one has the following parameter interchange transformations for a 
terminating ${}_7W_6^{-1}$:
\begin{eqnarray}
\label{cor:4.13-2:r1} &&
\hspace{-1cm}
\Wphyp{7}{6}{-1}{\frac{q^{-n}c}{d}}
{q^{-n},\frac{q^{-n}c}{a},\frac{qa}{de},c}
{q,\frac{q^ne}{c}}
\\[-0.05cm]
\label{cor:4.13-2:r2} &&\hspace{1cm}=
\frac{
(\frac{qa}{cd},\frac{qa}{e},\frac{d}{c},e;q)_n
}{
(\frac{qa}{ce},\frac{qa}{d},\frac{e}{c},d;q)_n
}
\Wphyp{7}{6}{-1}{\frac{q^{-n}c}{e}}
{q^{-n},\frac{q^{-n}c}{a},\frac{qa}{ed},c}
{q,\frac{q^nd}{c}}
\\[-0.05cm]
\label{cor:4.13-2:r3} &&\hspace{1cm}=
\frac{
(\frac{qa}{de},\frac{qa}{c},\frac{d}{c},c;q)_n
}{
(\frac{qa}{ce},\frac{qa}{d},\frac{c}{d},d;q)_n
}
\Wphyp{7}{6}{-1}{\frac{q^{-n}d}{c}}
{q^{-n},\frac{q^{-n}d}{a},\frac{qa}{ce},d}
{q,\frac{q^ne}{d}}
\\[-0.05cm]
\label{cor:4.13-2:r4} &&\hspace{1cm}=
\frac{
(\frac{qa}{cd},\frac{qa}{e},\frac{d}{c},e;q)_n
}{
(\frac{qa}{ce},\frac{qa}{d},\frac{e}{d},d;q)_n
}
\Wphyp{7}{6}{-1}{\frac{q^{-n}d}{e}}
{q^{-n},\frac{q^{-n}d}{a},\frac{qa}{ec},d}
{q,\frac{q^nc}{d}}
\\[-0.05cm]
\label{cor:4.13-2:r5} &&\hspace{1cm}=
\frac{
(\frac{qa}{de},\frac{qa}{c},\frac{d}{c},c;q)_n
}{
(\frac{qa}{ce},\frac{qa}{d},\frac{c}{e},d;q)_n
}
\Wphyp{7}{6}{-1}{\frac{q^{-n}e}{c} }
{q^{-n},\frac{q^{-n}e}{a} ,\frac{qa}{cd},e}
{q,\frac{q^nd}{e}}\\[-0.05cm]
\label{cor:4.13-2:r6} &&\hspace{1cm}=
\frac{(\frac{d}{c};q)_n}{(\frac{d}{e};q)_n}
\Wphyp{7}{6}{-1}{\frac{q^{-n}e}{d}}
{q^{-n},\frac{q^{-n}e}{a},\frac{qa}{dc},e}{q,\frac{q^nc}{e}}.
\end{eqnarray}
\label{cor:4.12}
\end{cor}
\begin{proof}
\noindent
Start with \eqref{cor:4.13-2} and consider all permutations of the symmetric 
parameters $c,d,e$ which produce non-trivial transformations.
\end{proof}

\begin{cor} Let $n\in\mathbb N_0$, $a,c,d,e\in\CCast$, 
$q\in\CCdag$.
%such that $|q|\ne 1$. 
Then, one has the following parameter interchange transformations for a 
terminating ${}_3\phi_2$:
\label{cor4.12}
\begin{eqnarray}
\label{cor:4.13-3:r1}
&&\hspace{-4cm}
\qhyp{3}{2}{q^{-n},\frac{qa}{cd},e}{\frac{qa}{c},\frac{qa}{d}}{q,\frac{q^{n+1}a}{e}}
=
\frac{(\frac{qa}{e};q)_n}
{(\frac{qa}{d};q)_n}
\!\qhyp{3}{2}{q^{-n},\frac{qa}{ce},d}{\frac{qa}{c},\frac{qa}{e}}{q,\frac{q^{n+1}a}{d}}
\\[-0.05cm]\label{cor:4.13-3:r4}&&
\hspace{0.5cm}=
\frac{(\frac{qa}{e};q)_n}
{(\frac{qa}{c};q)_n}
\!\qhyp{3}{2}{q^{-n},\frac{qa}{de},c}{\frac{qa}{d},\frac{qa}{e}}{q,\frac{q^{n+1}a}{c}}.
\end{eqnarray}
\end{cor}
\begin{proof}
\noindent
Start with \eqref{cor:4.13-3} and consider all permutations of the 
symmetric parameters $c,d,e$ which produce non-trivial transformations.
\end{proof}

\begin{rem}
Another set of parameter interchange transformations
can be obtained by considering all permutations
of the symmetric parameters $c,d,e$ in
\eqref{cor:4.13-4}
and
\eqref{cor:4.13-8}
respectively. However, one can see that
these are equivalent to the above Corollary 
\ref{cor4.12} by replacing
\begin{equation}
\hspace{-0.0cm}(a,c,d,e)\mapsto\left(\frac{q^{-2n}}{a},\frac{q^{-n}c}{a},
\frac{q^{-n}d}{a},\frac{q^{-n}e}{a}\right),
\end{equation}
\begin{equation}\hspace{-0.0cm}
(a,c,d,e)\mapsto \left(\frac{q^{1-n}cde}{a},c,d,e\right),
\end{equation}
respectively.
\end{rem}

\begin{cor}
\label{cor4.14}Let $n\in\mathbb N_0$, $a,c,d,e\in\CCast$, 
$q\in\CCdag$.
%such that $|q|\ne 1$. 
Then, one has the following parameter interchange transformations 
for a terminating ${}_3\phi_2$:
\begin{eqnarray}
\label{cor:4.13-5:r1}
&&\hspace{-0.5cm}
\qhyp{3}{2}{q^{-n},c,d}{\frac{q^{-n}cd}{a},\frac{qa}{e}}{q,q}
=
\frac
{
(\frac{qa}{ce},\frac{qa}{d};q)_n
}{
(\frac{qa}{cd},\frac{qa}{e};q)_n
}
\qhyp{3}{2}{q^{-n},c,e}{\frac{q^{-n}ce}{a},\frac{qa}{d}}{q,q}
=
\frac
{
(\frac{qa}{de},\frac{qa}{c};q)_n
}{
(\frac{qa}{cd},\frac{qa}{e};q)_n
}
\qhyp{3}{2}{q^{-n},d,e}{\frac{q^{-n}de}{a},\frac{qa}{c}}{q,q}.
\end{eqnarray}
\end{cor}
\begin{proof}
\noindent
Start with \eqref{cor:4.13-5} and consider all permutations of 
the symmetric parameters $c,d,e$ which produce non-trivial transformations.
\end{proof}

\begin{rem}
Another set of parameter interchange transformations
can be obtained by considering all permutations
of the symmetric parameters $c,d,e$ in
\eqref{cor:4.13-6} and \eqref{cor:4.13-7}
respectively. However, one can see that
these are equivalent to the above Corollary 
\ref{cor4.14} by replacing
\[
(a,c,d,e)\mapsto \left(\frac{q^{-2n}}{a},\frac{q^{-n}c}{a},\frac{q^{-n}d}{a},
\frac{q^{-n}e}{a}\right),
\]
\[
(a,c,d,e)\mapsto\left(\frac{q^{1-n}a}{cde},\frac{qa}{cd},
\frac{qa}{ce},\frac{qa}{de}\right),
\]
respectively.
\end{rem}

%%%%%%%%%%%%%%%%%%%%%%%%%%%%%%%%%%%%%%%%%%%%%%%%%%%%%%%%%%%%%%%%%%%%%%%%%%%%%%%%%%%
\section{The \textit{q} and \textit{q}\textsuperscript{-1}
\!\!-Al-Salam--Chihara polynomials}
\label{alsalamchihara}
%%%%%%%%%%%%%%%%%%%%%%%%%%%%%%%%%%%%%%%%%%%%%%%%%%%%%%%%%%%%%%%%%%%%%%%%%%%%%%%%%%%

Define the set ${\bf 2}:=\{1,2\}$, and the multiset ${\bf a}:=\{a_1,a_2\}$, $a_k\in\CCast$, 
$k\in{\bf 2}$, and $x=\frac12(z+z^{-1})$,
$q\in\CCdag$.
The Al-Salam--Chihara polynomials $Q_n(x;{\bf a}|q)$ are a family of polynomials
symmetric in two free parameters. These can be obtained from the Askey--Wilson 
polynomials by taking two of its free parameters to be zero, or from the 
continuous dual $q$-Hahn polynomials by taking one of its free parameters 
to be zero. Hence, the Al-Salam--Chihara polynomials are a symmetric subfamily 
of the Askey--Wilson polynomials. 
%%%%%%%%%%%%%%%%%%%%%%%%%%%%%%%%%%%%%%%%%%%%%%%%%%%%%%%%%%%%%%%%%%%%%%%%%%%%%%%%%%%
\subsection{Terminating $q$ and $q^{-1}$-Al-Salam--Chihara 
polynomial representations}
%%%%%%%%%%%%%%%%%%%%%%%%%%%%%%%%%%%%%%%%%%%%%%%%%%%%%%%%%%%%%%%%%%%%%%%%%%%%%%%%%%%

\noindent One has the following complete list of terminating ${}_6W_5^{\pm 2}$ and equivalent terminating basic hypergeometric representations for the Al-Salam--Chihara polynomials. 

\begin{cor} \label{cor:5.1}
Let $n\in\mathbb N_0$, 
$q\in\CCdag$, 
$p,r,s\in{\bf 2}$, $p,r$ distinct and fixed.
Then, the Al-Salam--Chihara polynomials are given by:
\begin{eqnarray}
\label{ASC:def1} 
&&\hspace{-4.3cm}Q_n(x;{\bf a}|q):=a_p^{-n}(a_{12};q)_n
\qphyp{3}{1}{1}{q^{-n}, a_pz^{\pm}}{a_{12}}{q,q}\\
\label{ASC:def2} &&\hspace{-2.4cm}= q^{-\binom{n}{2}} 
(-a_p)^{-n} (a_pz^{\pm};q)_n\qhyp22{q^{-n},\frac{q^{1-n}}{a_{12}}}
{\frac{q^{1-n}z^{\pm}}{a_p}}{q,\frac{qa_{12}}{a_p^2}} \\
\label{ASC:def5} &&\hspace{-2.4cm}= z^n (a_{12};q)_n \qhyp31{q^{-n}, 
\{a_sz\}}{a_{12}}{q,\frac{q^n}{z^2}} %\\
\end{eqnarray}
\begin{eqnarray}
\label{ASC:def4} &&\hspace{-2.15cm}=z^{-n}(a_p z;q)_n 
\qhyp{2}{1}{q^{-n},\frac{a_r}{z}}
{\frac{q^{1-n}}{a_pz}}{q,\frac{qz}{a_p}}\\
\label{ASC:def3} &&\hspace{-2.15cm}=z^{-n}(\{a_sz\};q)_n
\qphyp{2}{2}{-1}{q^{-n},\frac{q^{1-n}}{a_{12}}}
{\left\{\frac{q^{1-n}}{a_sz}\right\}}{q,q}\\
\label{ASC:def7} &&\hspace{-2.15cm} =a_p^{-n}\dfrac{\left(a_rz^{\pm};q\right)_n}
{\left(\frac{a_r}{a_p};q\right)_n} 
\Wphyp{6}{5}{-2}{q^{-n}z^2}
{q^{-n},\{a_sz\}}
{q,\frac{q^n}{a_{12}z^2}}\\
\label{ASC:def6} &&\hspace{-2.15cm} 
=a_p^{-n} 
\frac{(a_rz^\pm;q)_n}{(a_r/a_p;q)_n}
\Wphyp{6}{5}{2}{\frac{q^{-n}a_p}{a_r}}
{q^{-n}, a_p z^{\pm}}
{q,\frac{q^{2-n}}{a^2_{r}}}
.
\end{eqnarray}
\end{cor}
\begin{proof}
\eqref{ASC:def1} is derived by taking \eqref{cdqH:def1} and mapping $a_k \mapsto 0$, $k\ne p$ 
(see also \cite[(14.8.1)]{Koekoeketal}),
\eqref{ASC:def2} is derived 
by taking \eqref{cdqH:def2} and mapping 
$a_k \mapsto 0$, $k\ne p$,
\eqref{ASC:def5} is derived by taking \eqref{cdqH:def3} and mapping $a_t \mapsto 0$,
\eqref{ASC:def4} is derived by taking \eqref{cdqH:def3} or \eqref{cdqH:def4} and 
mapping $a_r \mapsto 0$ (see also \cite[(14.8.1)]{Koekoeketal}),
\eqref{ASC:def3} is derived by taking \eqref{cdqH:def4} and mapping $a_t \mapsto 0$, 
\eqref{ASC:def7} is derived by taking \eqref{cdqH:def6} and mapping $a_t \mapsto 0$
and 
\eqref{ASC:def6} is derived by taking \eqref{cdqH:def5} and mapping $a_k \mapsto 0$, $k\in{\bf 3}$.
Also, whenever necessary the indices
$k$ in $a_k$ should be relabeled such that
$k\in{\bf 2}$.
\end{proof}

\begin{rem} \label{rem:5.2}
It is interesting to note that for Al-Salam--Chihara polynomials (and also for $q^{-1}$-Al-Salam--Chihara polynomials, see Theorem \ref{cor:5.9} below), there are three 
separate representations of these polynomials where the symmetry in the parameters 
is evident:~\eqref{ASC:def3}, \eqref{ASC:def5}, \eqref{ASC:def6}. 
This is unlike the situation for the Askey--Wilson \cite[(19)]{CohlCostasSantos20b} and continuous dual 
$q$-Hahn polynomials \eqref{cdqH:def5}, where there are only one such representation each.
\end{rem} 

%%%%%%%%%%%%%%%%%%%%%%%%%%%%%%%%%%%%%%%%%%%%%%%%%%%%%%%%%%%%%%%%%%%%%%%%%%%%%%%%%%%
%\subsection{Terminating $q^{-1}$-Al-Salam--Chihara 
%polynomial representations}
%%%%%%%%%%%%%%%%%%%%%%%%%%%%%%%%%%%%%%%%%%%%%%%%%%%%%%%%%%%%%%%%%%%%%%%%%%%%%%%%%%%

\noindent One has the following complete list of terminating ${}_6W_5^{\pm 2}$ and equivalent terminating basic hypergeometric representations for the $q^{-1}$-Al-Salam--Chihara polynomials.

\begin{cor} \label{cor:5.9}
Let $Q_n(x;{\bf a}|q)$ and the respective parameters be defined as previously. 
Then, the $q^{-1}$-Al-Salam--Chihara polynomials are given by:
\begin{eqnarray}
\label{qiASC:1}&&\hspace{-1.9cm}Q_n(x;{\bf a}|q^{-1})=
q^{-\binom{n}{2}}
\left(\frac{-a_{12}}{a_p}\right)^n 
\left(\frac{1}{a_{12}}
;q\right)_n \qhyp31{q^{-n}, 
\frac{z^{\pm}}{a_p}
}{
\frac{1}{a_{12}}
}
{q,\frac{q^na_p^2}{a_{12}}} \\ \label{qiASC:2} &&\hspace{0.27cm}
=q^{-\binom{n}{2}}(-a_p)^{n} \left(\frac{z^{\pm}}{a_p};q\right)_n 
\qphyp{2}{2}{-1}{q^{-n},q^{1-n}a_{12}}
{q^{1-n}a_pz^{\pm}}{q,q} \\
\label{qiASC:5} &&\hspace{0.27cm}=q^{-\binom{n}{2}}(-a_{12}
z)^n \left(\frac{1}{a_{12}};q\right)_n 
\qphyp{3}{1}{1}{q^{-n},\{\frac{1}{a_sz}\}}
{\frac{1}{a_{12}}}{q,q}\\
\label{qiASC:3} &&\hspace{0.27cm}=
q^{-\genfrac{(}{)}{0pt}{}{n}{2}} (-a_p)^n 
\bigg( \frac{1}{a_pz};q\bigg)_n \qhyp{2}{1}{q^{-n}, 
\frac{z}{a_r}}
{q^{1-n}a_pz}{q, qa_rz} \\
 \label{qiASC:4} &&\hspace{0.27cm}=
q^{-2\binom{n}{2}}(a_{12}z)^n
%z^n 
\bigg(\left\{\frac{1}{a_sz}\right\};q\bigg)_n
\qhyp22{q^{-n}, q^{1-n}a_{12}}{\{q^{1-n}a_sz\}}
{q,qz^{2}} \\ 
&&\hspace{0.27cm}
=q^{-{\binom{n}{2}}}
(-a_r)^{n}
\frac{(\frac{z^{\pm}}{a_r};q)_n}
{(\frac{a_p}{a_r};q)_n}
\label{qiASC:7HSC} 
\Wphyp{6}{5}{-2}{\frac{q^{-n}a_r}{a_p}}
{q^{-n},\frac{z^{\pm}}{a_p}}
{q,q^na_p^2}\\
%%%%%
&&\hspace{0.27cm}=q^{-\genfrac{(}{)}{0pt}{}{n}{2}}(-a_{12}z)^n 
\frac{(\{\frac{z}{a_s}\};q)_n}{(z^{2};q)_n} 
%\nonumber \\ 
\label{qiASC:6} 
%\\&&\hspace{2cm}
%\times
\Wphyp{6}{5}{2}{\frac{q^{-n}}{z^2}}
{q^{-n},\left\{\frac{1}{a_sz}\right\}}
{q,\frac{q^{2-n}a_{12}}{z^2}}.
\end{eqnarray}
\end{cor}
\begin{proof}
Each $q^{-1}$ representation is derived from the corresponding 
representation by applying the map $q\mapsto 1/q$ and using \eqref{poch.id:3}.
\end{proof}

%%%%%%%%%%%%%%%%%%%%%%%%%%%%%%%%%%%%%%%%%%%%%%%%%%%%%%%%%%%%%%%%%%%%%%%%%%%%%%%%%%%
\subsection{Terminating \textit{2}-parameter $q$ and $q^{-1}$ symmetric transformations}
%%%%%%%%%%%%%%%%%%%%%%%%%%%%%%%%%%%%%%%%%%%%%%%%%%%%%%%%%%%%%%%%%%%%%%%%%%%%%%%%%%%

\noindent By studying the interrelations of the terminating basic hypergeometric functions in Corollary \ref{cor:5.1} we obtain the following transformation result for a terminating ${}_6W_5^{-2}$.

\begin{cor}
\label{cor:5.8}
Let $n\in\mathbb N_0$, $a,c,d\in\CCast$,
$q\in\CCdag$.
Then, one has 
the following transformation
formulas for a terminating 
${}_{6}W_5^{-2}$:
\begin{eqnarray}
&&\label{cor:5.8.1}\hspace{-0.5cm}
\Wphyp{6}{5}{-2}{a}
{q^{-n},c,d}
{q,\frac{q^{n}}{cd}}
\\[-0.05cm]
&&\label{cor:5.8.2}\hspace{1cm}=
q^{-\binom{n}{2}}\left(\frac{-1}{cd}\right)^n\!\!\dfrac{\left(qa,a, c, d;q\right)_n}
{(a;q)_{2n}\!\left(\frac{qa}{c},\frac{qa}{d};q\right)_n} 
\Wphyp{6}{5}{-2}{\frac{q^{-2n}}{a}}
{q^{-n},\frac{q^{-n}c}{a},\frac{q^{-n}d}{a}}
{q,\frac{q^{3n}a^2}{cd}}\\[-0.05cm]
&&\label{cor:5.8.3}\hspace{1cm}=\left(\frac{1}{c}\right)^{2n}\frac{\left(qa,d;q\right)_n}
{\left(\frac{qa}{c},\frac{d}{c};q\right)_n}
\Wphyp{6}{5}{2}{\frac{q^{-n}c}{d}}
{q^{-n},\frac{q^{-n}c}{a},c}
{q,\frac{q^2a}{d^2}}\\[-0.05cm] 
&&\label{cor:5.8.4}\hspace{1cm}
=\frac{\left(qa,\frac{qa}{cd};q\right)_n}{\left(\frac{qa}{c},\frac{qa}{d};q\right)_n}
\qhyp31{q^{-n},c,d}{\frac{q^{-n}cd}{a}}{q,\frac{1}{a}}\\[-0.05cm]
&&\label{cor:5.8.5}\hspace{1cm}=
q^{-2\binom{n}{2}}\left(\frac{1}{qa}\right)^n
\dfrac{\left(qa,\frac{qa}{cd}; q\right)_n}{\left(\frac{qa}{c},\frac{qa}{d};q\right)_n} 
\qhyp31{q^{-n},\frac{q^{-n}c}{a}, 
\frac{q^{-n}d}{a}}{\frac{q^{-n}cd}a}{q,q^{2n}a}\\[-0.05cm]
&&\label{cor:5.8.7}\hspace{1cm}=
\left(\frac{1}{c}\right)^n\frac{\left(qa,\frac{qa}{cd};q\right)_n}
{\left(\frac{qa}{c},\frac{qa}{d};q\right)_n}
\qphyp{3}{1}{1}{q^{-n},\frac{q^{-n}c}{a},c}
{\frac{q^{-n}cd}{a}}{q,q}\\[-0.05cm]
&&\label{cor:5.8.6}\hspace{1cm}=
q^{-\binom{n}{2}}\left(\frac{-1}{cd}\right)^n
\frac{\left(qa,c;q\right)_n}
{\left(\frac{qa}{d};q\right)_n}
\qhyp22{q^{-n},\frac{qa}{cd}}
{\frac{q^{1-n}}{c},\frac{qa}{c}}{q,\frac{qd}{c}}\\[-0.05cm]
&&\hspace{1cm}\label{cor:5.8.8}=
\left(\frac{1}{c}\right)^n
\frac{(qa;q)_n}
{\left(\frac{qa}{c};q\right)_n}
\qhyp{2}{1}{q^{-n},c}{\frac{qa}{d}}{q,\frac{q}{d}}
\\[-0.05cm]
&&\label{cor:5.8.9}\hspace{1cm}=
q^{-\binom{n}{2}}\left(\frac{-1}{cd}\right)^n
\frac{(qa,d;q)_n}
{\left(\frac{qa}{c},\frac{qa}{d}\right)_n}
\qhyp{2}{1}{q^{-n},\frac{q^{-n}c}{a}}{\frac{q^{1-n}}{d}}
{q,\frac{q^{n+1}a}{d}}\\[-0.05cm]
&&\label{cor:5.8.10}\hspace{1cm}=
q^{-\binom{n}{2}}\left(\frac{-1}{qa}\right)^n
(qa;q)_n
\qphyp{2}{2}{-1}{q^{-n},\frac{qa}{cd}}
{\frac{qa}{c},\frac{qa}{d}}{q,q}\\[-0.05cm]
&&\label{cor:5.8.11}\hspace{1cm}=
q^{-\binom{n}{2}}\left(\frac{-1}{cd}\right)^n
\frac{(qa,c,d;q)_n}{\left(\frac{qa}{c},\frac{qa}{d};q\right)_n}
\qphyp{2}{2}{-1}{q^{-n},\frac{qa}{cd}}
{\frac{q^{1-n}}{c},\frac{q^{1-n}}{d}}{q,q}.
\end{eqnarray}
\end{cor}
\begin{proof}
Start with Corollary \ref{cor:4.3}, then take the limit as $e\to 0$. 
This completes the proof.
\end{proof}

%%%%%%%%%%%%%%%%%%%%%%%%%%%%%%%%%%%%%%%%%%%%%
%\subsection{Terminating \textit{2}-parameter $q^{-1}$ symmetric transformations}
%%%%%%%%%%%%%%%%%%%%%%%%%%%%%%%%%%%%%%%%%%%%%

\noindent By studying the interrelations of the terminating basic hypergeometric functions in Corollary \ref{cor:5.9} we obtain the following transformation result for a terminating ${}_6W_5^{2}$.

\begin{cor}
\label{cor5.14b}
Let $n\in\mathbb N_0$, $a,c,d\in\CCast$,
$q\in\CCdag$.
%such that $|q|\ne 1$.
Then, one has the following transformation
formulas for a terminating
${}_6W_{5}^{2}$:
\begin{eqnarray}
&&\label{cor5.14.1}\hspace{-0.8cm}
\Wphyp{6}{5}{2}{a}{q^{-n},c,d}{q,\frac{q^{n+2}a^2}{cd}}
\\[-0.05cm]
&&\label{cor5.14.2}=
q^{3\binom{n}{2}}\left(\frac{-q^2a^2}{cd}\right)^n
\frac{(qa,a,c,d;q)_n}{(a;q)_{2n}(\frac{qa}{c},\frac{qa}{d};q)_n}
\Wphyp{6}{5}{2}{\frac{q^{-2n}}{a}}
{q^{-n},\frac{q^{-n}c}{a},\frac{q^{-n}d}{a}}
{q,\frac{q^{2-n}}{cd}}\\%[-0.00cm]
&&\label{cor5.14.3}=
\frac{(qa,d;q)_n}{\left(\frac{qa}{c},\frac{d}{c};q\right)_n}
\Wphyp{6}{5}{-2}{\frac{q^{-n}c}{d}}
{q^{-n},\frac{q^{-n}c}{a},c}
{q,\frac{q^{2n}a}{c^2}}\\
&&\label{cor5.14.9}=c^n\frac{\left(qa,\frac{qa}{cd};q\right)_n}{(\frac{qa}{c},\frac{qa}{d};q)_n}
\qhyp31{q^{-n},\frac{q^{-n}c}{a},c}{\frac{q^{-n}cd}{a}}{q,\frac{q^nd}{c}}
\\[-0.05cm]
&&\label{cor5.14.6}=\frac{\left(qa,\frac{qa}{cd};q\right)_n}{(\frac{qa}{c},\frac{qa}{d};q)_n}
\qphyp{3}{1}{1}{q^{-n},c,d}{\frac{q^{-n}cd}{a}}{q,q}
\\[-0.05cm]
&&\label{cor5.14.7}=q^{2\binom{n}{2}} (qa)^n \frac{\left(qa,\frac{qa}{cd};q\right)_n}{(\frac{qa}{c},\frac{qa}{d};q)_n}
\qphyp{3}{1}{1}{q^{-n},\frac{q^{-n}c}{a},\frac{q^{-n}d}{a}}{\frac{q^{-n}cd}{a}}{q,q}
\\[-0.05cm]
&&\label{cor5.14.4}=\left(qa;q\right)_n
\qhyp22{q^{-n},\frac{qa}{cd}}{\frac{qa}{c},\frac{qa}{d}}{q,q^{n+1}a}
\\[-0.05cm]
&&\label{cor5.14.5}=\left(\frac{qa}{cd}\right)^n
\frac{\left(qa,c,d;q\right)_n}
{\left(\frac{qa}{c},\frac{qa}{d};q\right)_n}
\qhyp22{q^{-n},\frac{qa}{cd}}{\frac{q^{1-n}}{c},\frac{q^{1-n}}{d}}{q,
\frac{q^{1-n}}{a}}
\\[-0.05cm]
&&\label{cor5.14.8}=
\frac{\left(qa,c;q\right)_n}
{\left(\frac{qa}{d};q\right)_n}
\qphyp{2}{2}{-1}{q^{-n},\frac{qa}{cd}}{\frac{qa}{c},\frac{q^{1-n}}{c}}{q,q}\\[-0.05cm]
&&\label{cor5.14.10}=
q^{\binom{n}{2}}\left(\frac{-qa}{c}\right)^n
\frac{(qa,c;q)_n}{\left(\frac{qa}{c},\frac{qa}{d};q\right)_n}
\qhyp{2}{1}{q^{-n},\frac{q^{-n}d}{a}}{\frac{q^{1-n}}{c}}{q,\frac{q}{d}}
\\[-0.05cm]
&&\label{cor5.14.11}=
\frac{(qa;q)_n}{(\frac{qa}{d};q)_n}
\qhyp{2}{1}{q^{-n},d}{\frac{qa}{c}}{q,\frac{q^{n+1}a}{d}}.
\end{eqnarray}
\end{cor}

\begin{proof}
Taking the limit as $e\to\infty$ in 
Corollary \ref{cor:4.13} completes the proof.
\end{proof}

%%%%%%%%%%%%%%%%%%%%%%%%%%%%%%%%%%%%%%%%%%%%%%%%%%%%%%%%%%%%%%%%%%%%%%%%%%%%%%%%%%%
\subsection{Terminating \textit{2}-parameter $q$ and $q^{-1}$ symmetric interchange transformations}
%%%%%%%%%%%%%%%%%%%%%%%%%%%%%%%%%%%%%%%%%%%%%%%%%%%%%%%%%%%%%%%%%%%%%%%%%%%%%%%%%%%
\begin{cor}
Let $n\in\mathbb N_0$, $a,c,d\in\CCast$, 
$q\in\CCdag$.
Then, one has 
the following parameter interchange transformations for a terminating 
${}_6W_{5}^{2}$:
\begin{eqnarray}
&&\hspace{-0.5cm}
\Wphyp{6}{5}{2}{\frac{q^{-n}c}{d}}
{q^{-n},\frac{q^{-n}c}{a},c}
{q,\frac{q^2a}{d^2}}
%\nonumber\\
%&&\hspace{2cm}
=\label{cor:5.8.3p}
\left(\frac{c}{d}\right)^{2n}\!
\frac{\left(\frac{qa}{c},\frac{d}{c},c;q\right)_n}
{\left(\frac{qa}{d},\frac{c}{d},d;q\right)_n}
\Wphyp{6}{5}{2}{\frac{q^{-n}d}{c}}
{q^{-n},\frac{q^{-n}d}{a},d}
{q,\frac{q^2a}{c^2}}.
\end{eqnarray}
\label{cor5.5}
\end{cor}
\begin{proof}
\noindent
Start with \eqref{cor:5.8.3} and consider all permutations of the symmetric parameters $c,d$ which produce non-trivial transformations.
\end{proof}

\begin{cor}
Let $n\in\mathbb N_0$, $a,c,d\in\CCast$, 
$q\in\CCdag$.
%such that $|q|\ne 1$. 
Then, one has the following parameter interchange transformations for a terminating ${}_2\phi_2$:
\begin{equation}
\label{cor:5.8.6p}
\qhyp22{q^{-n},\frac{qa}{cd}}
{\frac{q^{1-n}}{c},\frac{qa}{c}}{q,\frac{qd}{c}}
=
\frac{\left(\frac{qa}{d},d;q\right)_n}
{\left(\frac{qa}{c},c;q\right)_n}
\qhyp22{q^{-n},\frac{qa}{cd}}
{\frac{q^{1-n}}{d},\frac{qa}{d}}{q,\frac{qc}{d}}
\end{equation}
\label{cor5.6b}
\end{cor}
\begin{proof}
\noindent
Start with \eqref{cor:5.8.6} and consider all permutations of the symmetric parameters $c,d$ which produce non-trivial transformations.
\end{proof}
\begin{cor}
Let $n\in\mathbb N_0$, $a,c,d\in\CCast$, 
$q\in\CCdag$.
Then, one has the following parameter interchange transformations for a terminating ${}_{3}\phi_1^1$:
\begin{equation}
\label{cor:5.8.7p}
\qphyp{3}{1}{1}{q^{-n},\frac{q^{-n}c}{a},c}{\frac{q^{-n}cd}{a}}{q,q}
=
\left(\frac{c}{d}\right)^n
\qphyp{3}{1}{1}{q^{-n},\frac{q^{-n}d}{a},d}{\frac{q^{-n}cd}{a}}{q,q}
\end{equation}
\label{cor5.7}
\end{cor}
\begin{proof}
\noindent
Start with \eqref{cor:5.8.7} and consider all permutations of the symmetric parameters $c,d$ which produce non-trivial transformations.
\end{proof}

\begin{cor}
\label{cor5.8}
Let $n\in\mathbb N_0$, $a,c,d\in\CCast$, 
$q\in\CCdag$.
%such that $|q|\ne 1$. 
Then, one has the following parameter interchange transformations for a terminating ${}_2\phi_1$:
\begin{equation}
\label{cor:5.8.8p}
\qhyp{2}{1}{q^{-n},c}{\frac{qa}{d}}{q,\frac{q}{d}}
=
\left(\frac{c}{d}\right)^n
\frac{\left(\frac{qa}{c};q\right)_n}
{\left(\frac{qa}{d};q\right)_n}
\qhyp{2}{1}{q^{-n},d}{\frac{qa}{c}}{q,\frac{q}{c}}
\end{equation}
\end{cor}
\begin{proof}
\noindent
Start with \eqref{cor:5.8.8} and consider all permutations of the symmetric parameters $c,d$ which produce non-trivial transformations.
\end{proof}

\begin{rem}
Another set of parameter interchange transformations
can be obtained by considering all permutations
of the symmetric parameters $c,d$ in
\eqref{cor:5.8.9}.
However, one can see that
these are equivalent to the above Corollary 
\ref{cor5.8} by replacing
\[
(a,c,d)\mapsto \left(\frac{q^{-2n}}{a},\frac{q^{-n}c}{a},\frac{q^{-n}d}{a}
\right).
\]
\end{rem}

%%%%%%%%%%%%%%%%%%%%%%%%%%%%%%%%%%%%%%%%%%%%%
%\subsection{Terminating 
%\textit{2}-parameter $q^{-1}$ symmetric interchange transformations}
%%%%%%%%%%%%%%%%%%%%%%%%%%%%%%%%%%%%%%%%%%%%%
\begin{cor}
Let $n\in\mathbb N_0$, $a,c,d\in\CCast$, 
$q\in\CCdag$.
%such that $|q|\ne 1$. 
Then, 
one has the following parameter interchange transformations 
for a terminating ${}_{6}W_5^{-2}$:
\begin{eqnarray}
&&\hspace{-0.8cm}
\Wphyp{6}{5}{-2}{\frac{q^{-n}c}{d}}
{q^{-n},\frac{q^{-n}c}{a},c}
{q,\frac{q^{2n}a}{c^2}}
%\\[-0.05cm]
%&&
\label{cor5.14.3p}
%\hspace{3cm}
=\frac{(\frac{qa}{c},\frac{d}{c},c;q)_n}
{\left(\frac{qa}{d},\frac{c}{d},d;q\right)_n}
\Wphyp{6}{5}{-2}{\frac{q^{-n}d}{c}}
{q^{-n},\frac{q^{-n}d}{a},d}{q,\frac{q^{2n}a}{d^2}}.
\end{eqnarray}
\label{cor5.14}
\end{cor}
\begin{proof}
\noindent 
Start with \eqref{cor5.14.3} and consider all permutations of the symmetric parameters $c,d$ which produce non-trivial transformations.
\end{proof}

\begin{cor}
Let $n\in\mathbb N_0$, $a,c,d\in\CCast$, 
$q\in\CCdag$.
%such that $|q|\ne 1$. 
Then, one has the following parameter interchange transformations for a terminating ${}_{2}\phi_{2}^{-1}$:
\begin{equation}
\label{cor5.14.8p}
\qphyp{2}{2}{-1}{q^{-n},\frac{qa}{cd}}{\frac{qa}{c},\frac{q^{1-n}}{c}}{q,q}
=
\frac{\left(\frac{qa}{d},d;q\right)_n}
{\left(\frac{qa}{c},c;q\right)_n}
\qphyp{2}{2}{-1}{q^{-n},\frac{qa}{cd}}{\frac{qa}{d},\frac{q^{1-n}}{d}}{q,q}.
\end{equation}
\label{cor5.15}
\end{cor}
\begin{proof}
\noindent 
Start with \eqref{cor5.14.8} and consider all permutations of the symmetric parameters $c,d$ which produce non-trivial transformations.
\end{proof}

\begin{cor}
Let $n\in\mathbb N_0$, $a,c,d\in\CCast$, 
$q\in\CCdag$.
%such that $|q|\ne 1$. 
Then, one has the following parameter interchange transformations for a terminating ${}_3\phi_1$:
\begin{equation}
\label{cor5.14.9p}
\qhyp31{q^{-n},\frac{q^{-n}c}{a},c}{\frac{q^{-n}cd}{a}}{q,\frac{q^nd}{c}}=\left(\frac{d}{c}\right)^n
\qhyp31{q^{-n},\frac{q^{-n}d}{a},d}{\frac{q^{-n}cd}{a}}{q,\frac{q^nc}{d}}.
\end{equation}
\label{cor5.16}
\end{cor}
\begin{proof}
\noindent 
Start with \eqref{cor5.14.9} and consider all permutations of the symmetric parameters $c,d$ which produce non-trivial transformations.
\end{proof}

\begin{cor}
Let $n\in\mathbb N_0$, $a,c,d\in\CCast$, 
$q\in\CCdag$.
%such that $|q|\ne 1$. 
Then, one has the following parameter interchange transformations for a terminating ${}_2\phi_1$:
\begin{equation}
\label{cor5.14.10p}
\qhyp{2}{1}{q^{-n},\frac{q^{-n}d}{a}}{\frac{q^{1-n}}{c}}{q,\frac{q}{d}}=
\left(\frac{c}{d}\right)^n\frac{(d;q)_n}{(c;q)_n}
\qhyp{2}{1}{q^{-n},\frac{q^{-n}c}{a}}{\frac{q^{1-n}}{d}}{q,\frac{q}{c}}.
\end{equation}
\label{cor5.17}
\end{cor}
\begin{proof}
\noindent
Start with \eqref{cor5.14.10} and consider all permutations of the symmetric parameters $c,d$ which produce non-trivial transformations.
\end{proof}

%%%%%%%%%%%%%%%%%%%%%%%%%%%%%%%%%%%%%%%%%%%%%
\subsection{Terminating \textit{2}-parameter
summation formulas}
\label{twoparamsum}
%%%%%%%%%%%%%%%%%%%%%%%%%%%%%%%%%%%%%%%%%%%%%
By taking the limits of the continuous dual $q$-Hahn 
and the continuous dual $q^{-1}$-Hahn polynomial representations
and transformations as one of the parameters tends to either zero
or infinity, one can obtain summation formulas for a ${}_6W_5$
\cite[(II.21)]{GaspRah}.

\begin{thm}
Let $p_n(x;{\bf a}|q)$ be the continuous 
dual $q$-Hahn polynomial, with $n\in\mathbb N_0$,
${\bf a}=\{a_1,a_2,a_3\}$,
$a_k\in\CC^\ast$,
$k,p\in{\bf 3}$.
Then
\begin{equation}
\lim_{a_p\to\infty}\frac{1}{a_p^n}p_n(x;{\bf a}|q)
=X_n({\bf a}_{[p]}|q)
=q^{\binom{n}{2}}(-1)^n(a_{12};q)_n,
\end{equation}
\begin{equation}
\lim_{a_p\to\infty}
\frac{1}{a_p^n}p_n(x;{\bf a}|q^{-1})
=q^{-3\binom{n}{2}}(-a_{12})^n
X_n({\bf a}_{[p]}^{-1}|q)
=q^{-2\binom{n}{2}}a_{12}^n(a_{12}^{-1};q)_n,
\end{equation}
where
\begin{equation}
\label{Xnform}
X_n({\bf a}_{[p]}|q):=q^{\binom{n}{2}}(-1)^n
\frac{(\{\frac{a_s}{z}\};q)_n}{(z^{-2};q)_n}
\Whyp65{q^{-n}z^2}{q^{-n},\{a_sz\}}{q,\frac{q}{a_{12}}},
\end{equation}
and the indices of ${\bf a}_{[p]}$ have been relabeled such that they are in ${\bf 2}$.
\end{thm}

\begin{proof}
Applying the necessary
limits using Lemmas \ref{leminftyWhyp} 
completes the proof. 
\end{proof}

\begin{cor}{\cite[(II.21)]{GaspRah}}
\label{cor5.6}
Let $n\in\mathbb N_0$, $a,c,d\in\CCast$, 
$q\in\CCdag$.
%such that $|q|\ne 1$. 
Then, one 
has the following transformation and summation formulas for a terminating 
${}_{6}W_{5}$:
\begin{eqnarray} 
&&\Whyp{6}{5}{a}{q^{-n},c,d}{q,\frac{q^{n+1}a}{cd}}\nonumber\\
&&\hspace{1cm}=
q^{\binom{n}{2}}\left(\frac{-qa}{cd}\right)^n\frac{(qa,a,c,d;q)_n}
{(a;q)_{2n}(\frac{qa}{c},\frac{qa}{d};q)_n}\Whyp{6}{5}{\frac{q^{-2n}}{a}}
{q^{-n},\frac{q^{-n}c}{a},\frac{q^{-n}d}{a}}
{q,\frac{q^{n+1}a}{cd}}\nonumber\\
&&\hspace{1cm}=\frac{(qa,\frac{qa}{cd};q)_n}
{(\frac{qa}{c},\frac{qa}{d};q)_n}
.
\end{eqnarray}
\end{cor}
\begin{proof}
Starting with Corollaries \ref{cor:4.3},
\ref{cor:4.13} and letting 
$e\to\infty,0$ respectively
using Lemma \ref{leminftyWhyp} 
completes the proof.
\end{proof}

%%%%%%%%%%%%%%%%%%%%%%%%%%%%%%%%%%%%%%%%%%%%%
\section{The continuous big \textit{q} and \textit{q}\textsuperscript{-1}
\!\!-Hermite polynomials}
\label{cbqH}
%%%%%%%%%%%%%%%%%%%%%%%%%%%%%%%%%%%%%%%%%%%%%
The continuous big $q$-Hermite polynomials $H_n(x;a|q)$ are a family 
of polynomials which can be obtained from the Askey--Wilson 
polynomials by taking three of its free parameters to be zero. It can
also be obtained from the continuous dual $q$-Hahn polynomials by taking two 
of its free parameters to be zero, or from the Al-Salam--Chihara 
polynomials by taking one of its parameters to be zero.
%%%%%%%%%%%%%%%%%%%%%%%%%%%%%%%%%%%%%%%%%%%%%
\subsection{Terminating continuous big \textit{q} and \textit{q}\textsuperscript{-1}
\!\!-Hermite polynomial representations}
%%%%%%%%%%%%%%%%%%%%%%%%%%%%%%%%%%%%%%%%%%%%%

\noindent One has the following complete list of terminating ${}_5W_4^{-3}$ and equivalent terminating basic hypergeometric representations for the continuous big $q$-Hermite polynomials. 

\begin{cor} 
\label{corcbqH}
Let 
$n\in\mathbb N_0$, 
$q\in\CCdag$,
$a, z\in\CCast$,
$x=\frac12(z+z^{-1})\in\CCast$.
The continuous big $q$-Hermite polynomials are given by:
\begin{eqnarray}
\label{cbqH:def1}
&&\hspace{-5.9cm}H_n(x;a|q)
:=a^{-n} \qphyp{3}{0}{2}{q^{-n}, az^\pm }{-}{q,q}\\
\label{cbqH:def2} &&\hspace{-4cm}= 
q^{-\binom{n}{2}} (-a)^{-n}
(az^\pm;q)_n 
\qhyp12{q^{-n}}{\frac{q^{1-n}z^\pm}{a}}{q,\frac{q^{2-n}}{a^2}} \\
\label{cbqH:def3}
&&\hspace{-4cm}=z^{-n}(az;q)_n 
\qphyp{1}{1}{-1}{q^{-n}}{\frac{q^{1-n}}{az}}{q,\frac{qz}{a}}\\
\label{cbqH:def3b}
&&\hspace{-4cm}=z^n
\frac{(\frac{a}{z};q)_n}
{(\frac{q}{az};q)_\infty}
\qhyp11{\frac{qz}{a}}{\frac{q^{1-n}z}{a}}{q,\frac{q^{1-n}}{az}}\\
\label{cbqH:def4}
&&\hspace{-4cm}=z^n \qhyp{2}{0}{q^{-n}, az }{-}{q,\frac{q^n}{z^2}}\\
\label{cbqH:def5}
\hspace{-0.9cm}
&&\hspace{-4cm}=
z^n
\frac{(\frac{a}{z};q)_n}{(z^{-2};q)_n}
\Wphyp{5}{4}{-3}{q^{-n}z^2}
{q^{-n}, az}
{q,\frac{q^{2n-1}}{az^3}}.
\end{eqnarray}
\end{cor}

%\noindent\noro{[HSC: (6.5) $\to 8\phi4$, but a different one! Maybe coincidence.]}

\begin{proof}
\eqref{cbqH:def1} is derived by taking \eqref{ASC:def1} and mapping $a_2 \mapsto 0$
(see also \cite[(14.18.1)]{Koekoeketal}),
\eqref{cbqH:def2} is derived by taking \eqref{ASC:def2} and mapping $a_2 \mapsto 0$, 
\eqref{cbqH:def3} is derived by taking \eqref{ASC:def3} and mapping $a_2 \mapsto 0$, 
\eqref{cbqH:def3b} is derived from
\eqref{cbqH:def3} using 
\cite[(1.13.6)]{Koekoeketal},
\eqref{cbqH:def4} is derived by taking \eqref{ASC:def4} or \eqref{ASC:def5} and 
mapping $a_r \mapsto 0$
(see also \cite[(14.18.1)]{Koekoeketal}),
\eqref{cbqH:def5} is derived by taking by taking \eqref{ASC:def6} and 
mapping $a_2 \mapsto 0$. 
\end{proof}
%%%%%%%%%%%%%%%%%%%%%%%%%%%%%%%%%%%%%%%%%%%%%%%%%%%%%%%%%%%%%%%%%%%%%%%%%%%%%%%%%%%
%\subsection{Terminating continuous big $q^{-1}$-Hermite polynomial representations}
%%%%%%%%%%%%%%%%%%%%%%%%%%%%%%%%%%%%%%%%%%%%%%%%%%%%%%%%%%%%%%%%%%%%%%%%%%%%%%%%%%%

\noindent One has the following complete list of terminating ${}_5W_4^{3}$ and equivalent terminating basic hypergeometric representations for the continuous big $q^{-1}$-Hermite polynomials. 

\begin{cor}
\label{corcbqiH}
Let $H_n(x;a|q)$ and the respective parameters be defined as previously. 
Then, the continuous big $q^{-1}$-Hermite polynomials are given by:
\begin{eqnarray}
\label{cbqiH:1} &&\hspace{-4.3cm}H_n(x;a|q^{-1}) =a^{-n}\qhyp{3}{0}{q^{-n},\frac{z^{\pm}}
{a}}{-}{q,q^na^2}\\ \label{cbqiH:2}&&\hspace{-2.0cm}=q^{-\binom{n}{2}}(-a)^{n} 
\left(\frac{z^{\pm}}{a};q\right)_n \qphyp{1}{2}{-2}{q^{-n}}
{q^{1-n}az^{\pm}}{q,q}\\ \label{cbqiH:3}&&\hspace{-2.0cm}=
q^{-\genfrac{(}{)}{0pt}{}{n}{2}}(-a)^n
\left(\frac{1}{az};q\right)_n \qhyp{1}{1}{q^{-n}}
{q^{1-n}az}{q,qz^{2}}\\ \label{cbqiH:4} &&\hspace{-2.0cm}
=z^n\qphyp{2}{0}{1}{q^{-n},\frac{1}{az}}{-}
{q,\frac{qa}{z}}\\ 
&&\hspace{-2.0cm}=(a z^{2})^n \frac{(\frac{z}
{a};q)_n}{(z^{2};q)_n}
\Wphyp{5}{4}{3}{\frac{q^{-n}}{z^2}}
{q^{-n},{\frac{1}{az}}}
{q,\frac{q^{2-n}a}{z^3}}.
\label{cbqiH:5} 
\end{eqnarray}
\end{cor}
\begin{proof}
Each $q^{-1}$ representation is derived from the corresponding representation 
by applying the map $q\mapsto 1/q$ and using \eqref{poch.id:3}.
\end{proof}

%%%%%%%%%%%%%%%%%%%%%%%%%%%%%%%%%%%%%%%%%%%%%
\subsection{Terminating \textit{1}-parameter $q$ and $q^{-1}$ transformations}
%%%%%%%%%%%%%%%%%%%%%%%%%%%%%%%%%%%%%%%%%%%%%

\noindent By studying the interrelations of the terminating basic hypergeometric functions in Corollary \ref{corcbqH} we obtain the following transformation result for a terminating ${}_5W_4^{-3}$.

\begin{cor}
\label{cor6.6a}
Let $n\in\mathbb N_0$, $a,c\in\CCast$, 
$q\in\CCdag$.
%such that $|q|\ne 1$.
Then, one has the following transformation formulas for a terminating 
${}_{5}W_4^{-3}$:
\begin{eqnarray}
&&\hspace{-1cm}
\Wphyp{5}{4}{-3}{a}{q^{-n},c}{q,\frac{q^{n-1}}{ac}}\\ 
&&\hspace{0cm}=q^{-2\binom{n}{2}}\left(\frac{1}
{qac}\right)^n\!\!\dfrac{\left(qa,a, c;q\right)_n}
{(a;q)_{2n}\!\left(\frac{qa}{c};q\right)_n} 
\Wphyp{5}{4}{-3}{\frac{q^{-2n}}{a}}{q^{-n},\frac{q^{-n}c}{a}}
{q,\frac{q^{4n-1}a^2}{c}}\\
&&\hspace{0cm}=\left(\frac{1}{c}\right)^{2n}\frac{\left(qa;q\right)_n}
{\left(\frac{qa}{c};q\right)_n}\qphyp{3}{0}{2}{q^{-n},\frac{q^{-n}c}{a},c}{-}{q,q}\\
&&\hspace{0cm}=q^{-2\binom{n}{2}}
\left(\frac{1}{qac}\right)^n
\dfrac{\left(qa; q\right)_n}{\left(\frac{qa}{c};q\right)_n} 
\qhyp{2}{0}{q^{-n},\frac{q^{-n}c}{a}, }{-}{q,q^{2n}a}\\
&&\hspace{0cm}=\left(\frac{1}{c}\right)^n\frac{\left(qa;q\right)_n}
{\left(\frac{qa}{c};q\right)_n}\qhyp{2}{0}{q^{-n},c}{-}{q,\frac{1}{a}}\\
&&\hspace{0cm}=
q^{-2\binom{n}{2}}\left(\frac{1}{qac}\right)^n\left(qa,c;q\right)_n
\qhyp12{q^{-n}}{\frac{q^{1-n}}{c},\frac{qa}{c}}{q,\frac{q^2a}{c^2}}\\
&&\hspace{0cm}=q^{-\binom{n}{2}}\left(\frac{-1}{qa}\right)^n
(qa;q)_n\qphyp{1}{1}{-1}{q^{-n}}{\frac{qa}{c}}{q,\frac{q}{c}}\\
&&\hspace{0cm}=q^{-2\binom{n}{2}}\left(\frac{1}{qac}\right)^n
\frac{(qa,c;q)_n}{(\frac{qa}{c};q)_n}\qphyp{1}{1}{-1}{q^{-n}}{\frac{q^{1-n}}{c}}{q,\frac{q^{n+1}a}{c}}
\end{eqnarray}
\end{cor}
\begin{proof}
Start with Corollary \ref{cor:5.8}, then take the limit as $d\to 0$. This
completes the proof.
\end{proof}

%%%%%%%%%%%%%%%%%%%%%%%%%%%%%%%%%%%%%%%%%%%%%%%%%%%%%%%%%%%%%%%%%%%%%%%%%%%%%%%%%%%
%\subsection{Terminating \textit{1}-parameter $q^{-1}$ transformations}
%%%%%%%%%%%%%%%%%%%%%%%%%%%%%%%%%%%%%%%%%%%%%%%%%%%%%%%%%%%%%%%%%%%%%%%%%%%%%%%%%%%

\noindent By studying the interrelations of the terminating basic hypergeometric functions in Corollary \ref{corcbqiH} we obtain the following transformation result for a terminating ${}_5W_4^{3}$.

\begin{cor}
\label{cor6.11}
Let $n\in\mathbb N_0$, $a,c\in\CCast$, $q\in\CCdag$. 
Then, one has 
the following transformation
formulas for a terminating ${}_5W_{4}^3$:
\begin{eqnarray}
&&\hspace{-0.8cm}
\label{cor6.11.1}
\Wphyp{5}{4}{3}{a}{q^{-n},c}{q,\frac{q^{n+2}a^2}{c}}
\\ &&=\label{cor6.11.2}q^{4\binom{n}{2}}\left(\frac{q^2a^2}{c}\right)^n
\frac{(qa,a,c;q)_n}{(a;q)_{2n}(\frac{qa}{c};q)_n}
\Wphyp{5}{4}{3}{\frac{q^{-2n}}{a}}{q^{-n},\frac{q^{-n}c}{a}}
{q,\frac{q^{2-2n}}{ac}}\\ 
&&=c^n\frac{\left(qa;q\right)_n}{(\frac{qa}{c};q)_n}
\qhyp{3}{0}{q^{-n},\frac{q^{-n}c}{a},c}{-}{q,\frac{q^{2n}a}{c^2}}\\
&&=q^{2\binom{n}{2}}(qa)^n\frac{\left(qa;q\right)_n}{(\frac{qa}{c};q)_n}
\qphyp{2}{0}{1}{q^{-n},\frac{q^{-n}c}{a}}{-}{q,\frac{q}{c}}\\
&&=\frac{\left(qa;q\right)_n}{(\frac{qa}{c};q)_n}
\qphyp{2}{0}{1}{q^{-n},c}{-}{q,\frac{q^{n+1}a}{c}}\\
&&=\left(qa,c;q\right)_n
\qphyp{1}{2}{-2}{q^{-n}}{\frac{qa}{c},\frac{q^{1-n}}{c}}{q,q}\\
&&\label{cor6.11.4}=q^{\binom{n}{2}}\left(\frac{-qa}{c}\right)^n
\frac{\left(qa,c;q\right)_n}{(\frac{qa}{c};q)_n}
\qhyp{1}{1}{q^{-n}}{\frac{q^{1-n}}{c}}{q,\frac{q^{1-n}}{a}}\\
&& \label{cor6.11.3}
=\left(qa;q\right)_n\qhyp{1}{1}{q^{-n}}{\frac{qa}{c}}{q,q^{n+1}a}.
\end{eqnarray}
\end{cor}
\begin{proof}
Taking the limit as $d\to\infty$ in Corollary \ref{cor5.14} completes 
the proof. 
\end{proof}

%%%%%%%%%%%%%%%%%%%%%%%%%%%%%%%%%%%%%%%%%%%%%
\subsection{Terminating \textit{1}-parameter summation formulas}
\label{oneparamsum}
%%%%%%%%%%%%%%%%%%%%%%%%%%%%%%%%%%%%%%%%%%%%%
By taking the limits of the Al-Salam--Chihara, $X_n({\bf a}|q^{\pm 1})$
(see Section \ref{twoparamsum}), and the $q^{-1}$-Al-Salam--Chihara polynomial 
representations and transformations as one of the parameters tends to either zero
or infinity, one can obtain summation formulas for ${}_{5}W_4^{-1}$ and 
${}_5W_{4}^1$. Here we present these summation formulas (see Figure \ref{figure1}).

\begin{thm}
Let $Q_n(x;{\bf a}|q)$ be the Al-Salam--Chihara polynomial, $n\in\mathbb N_0$,
${\bf a}=\{a_1,a_2\}$, $a_k\in\CC^\ast$,
$k,p\in{\bf 2}$.
Then
\begin{equation}
\lim_{a_p\to\infty}\frac{1}{a_p^n}Q_n(x;{\bf a}|q)
=\lim_{a_p\to 0}X_n({\bf a}|q)
=Y_n^-({\bf a}_{[p]}|q)
=q^{\binom{n}{2}}(-1)^n,
\end{equation}
\begin{equation}
\lim_{a_p\to\infty}\frac{1}{a_p^n}X_n({\bf a}|q)
=Y_n^+({\bf a}_{[p]}|q)
=q^{2\binom{n}{2}}a^n,
\end{equation}
\begin{equation}
\lim_{a_p\to \infty}\frac{1}{a_p}
Q_n(x;{\bf a}|q^{-1})
=q^{-3\binom{n}{2}}(-a)^nY_n^+({\bf a}_{[p]}^{-1}|q)
=q^{-\binom{n}{2}}(-1)^n,
\end{equation}
%\[
%\lim_{a_p\to\infty}
%\frac{1}{a_p^n}Q_n(x;{\bf a}|q^{-1})
%=q^{-3\binom{n}{2}}(-a_{12})^n
%%Y_n^-({\bf a}_{[p]}^{-1}|q)
%=q^{-2\binom{n}{2}}a_{12}^n(a_{12}^{-1};q)_n,
%\]
where
\begin{equation}
Y_n^-({\bf a}_{[p]}|q):=q^{\binom{n}{2}}(-1)^n
\frac{(\frac{a}{z};q)_n}{(z^{-2};q)_n}
\Wphyp{5}{4}{-1}{q^{-n}z^2}{q^{-n},az}
{q,\frac{q^n}{az}},
\label{Ynminform}
\end{equation}
\begin{equation}
Y_n^+({\bf a}_{[p]}|q):=q^{2\binom{n}{2}}
z^{-n}
\frac{(\frac{a}{z};q)_n}{(z^{-2};q)_n}
\Wphyp{5}{4}{1}{q^{-n}z^2}{q^{-n},az}
{q,\frac{qz}{a}},
\label{Ynplusform}
\end{equation}
$X_n$ is given in \eqref{Xnform} and we have relabeled ${\bf a}_{[p]}$ as $a$ where necessary.
\end{thm}
\begin{proof}
Applying the necessary
limits using Lemma \ref{leminftyWhyp} 
completes the proof. 
\end{proof}

\begin{cor}
\label{cor6.4}
Let $n\in\mathbb N_0$, $a,c\in\CCast$,
$q\in\CCdag$.
Then, one has the following transformation
and summation formulas for a terminating ${}_{5}W_4^{-1}$:
\begin{eqnarray}
&&\Wphyp{5}{4}{-1}{a}{q^{-n},c}{q,\frac{q^n}{c}}
=\left(\frac{1}{c}\right)^n
\frac{(qa,a,c;q)_n}
{(a;q)_{2n}(\frac{qa}{c};q)_n}
\Wphyp{5}{4}{-1}
{\frac{q^{-2n}}{a}}{q^{-n},\frac{q^{-n}c}{a}}
{q,\frac{q^{2n}a}{c}}\nonumber\\
&&\hspace{4.0cm}=\left(\frac{1}{c}\right)^n\frac{(qa;q)_n}{(\frac{qa}{c};q)_n}.
\end{eqnarray}
\end{cor}
\begin{proof}
Starting with Corollary \ref{cor:5.8} and letting 
$d\to\infty$ using Lemma \ref{leminftyWhyp}
completes the proof. 
One may also obtain the identical result
by starting with Corollary \ref{cor5.6},
and taking the limit as $d\to0$.
This is indicated in 
Figure \ref{figure1} below in the box for
$Y_n^{-}$ which is given in terms of a terminating
${}_{5}W_4^{-1}$.
\end{proof}

\begin{cor}
\label{cor6.8}
Let $n\in\mathbb N_0$, $a,c\in\CCast$, $q\in\CCdag$.
%such that $|q|\ne 1$.
Then, one has 
the following transformation
and summation 
formulas for a terminating 
${}_{5}W_{4}^{1}$:
\begin{eqnarray}
&&\Wphyp{5}{4}{1}{a}{q^{-n},c}{q,\frac{q^{n+1}a}{c}}
=q^{2\binom{n}{2}}\left(\frac{qa}{c}\right)^n
\frac{(qa,a,c;q)_n}
{(a;q)_{2n}(\frac{qa}{c};q)_n}
\Wphyp{5}{4}{1}
{\frac{q^{-2n}}{a}}{q^{-n},\frac{q^{-n}c}{a}}
{q,\frac{q}{c}}\nonumber\\
&&\hspace{4.5cm}=\frac{(qa;q)_n}{(\frac{qa}{c};q)_n}.
\end{eqnarray}
\end{cor}
\begin{proof}
Starting with Corollary \ref{cor5.14} and letting 
$d\to0$ using Lemma \ref{leminftyWhyp}
completes the proof.
One may also obtain the identical
result by starting with Corollary \ref{cor5.6}
and taking the limit as $d\to\infty$.
This is indicated in 
Figure \ref{figure1} below in the box for
$Y_n^{+}$ which is given in terms of a terminating
${}_{5}W_{4}^{1}$.
\end{proof}

%%%%%%%%%%%%%%%%%%%%%%%%%%%%%%%%%%%%%%%%%%%%%%%%%%%%%%%%%%%%%%%%%%%%%%%%%%%%%%%%%%%
\section{The continuous \textit{q} and \textit{q}\textsuperscript{-1}
\!\!-Hermite polynomials}
\label{cqH}
%%%%%%%%%%%%%%%%%%%%%%%%%%%%%%%%%%%%%%%%%%%%%%%%%%%%%%%%%%%%%%%%%%%%%%%%%%%%%%%%%%%
The continuous $q$-Hermite polynomials $H_n(x|q)$ are a family 
of polynomials which can be obtained from the Askey--Wilson,
continuous dual $q$-Hahn, Al-Salam--Chihara or continuous big
$q$-Hermite polynomials by taking all of their free parameters 
to be zero. 

%%%%%%%%%%%%%%%%%%%%%%%%%%%%%%%%%%%%%%%%%%%%%%%%%%%%%%%%%%%%%%%%%%%%%%%%%%%%%%%%%%%
\subsection{Terminating continuous \textit{q} and $q^{-1}$-Hermite polynomial 
representations}
%%%%%%%%%%%%%%%%%%%%%%%%%%%%%%%%%%%%%%%%%%%%%%%%%%%%%%%%%%%%%%%%%%%%%%%%%%%%%%%%%%%

\noindent One has the following list of terminating ${}_4W_3^{-4}$ and equivalent terminating basic hypergeometric representations for the continuous $q$-Hermite polynomials. 

\begin{cor} 
\label{corcqH}
Let $n\in\mathbb N_0$, $q\in\CCdag$,
$z\in\CCast$,
$x=\frac12(z+z^{-1})\in\CCast$.
The continuous $q$-Hermite polynomials are given by:
\begin{eqnarray}
\label{cqH:def1}
&&\hspace{-6.35cm}H_n(x|q)
=z^n \qphyp{1}{0}{-1}{q^{-n}}{-}{q,\frac{q^n}{z^2}}\\
\label{cqH:def2}
&&\hspace{-4.9cm}=\frac{z^n}{(z^{-2};q)_n}
\Wphyp{4}{3}{-4}{q^{-n}z^2}{q^{-n}}
{q,\frac{q^{3n-2}}{z^4}}.
\end{eqnarray}
\end{cor}
\begin{proof}
For \eqref{cqH:def1}, start with \eqref{cbqH:def2}, \eqref{cbqH:def3}, or 
\eqref{cbqH:def4} and take the limit $a\to 0$
(see also \cite[(14.26.1)]{Koekoeketal}). For \eqref{cqH:def2}, take the 
limit as $a\to 0$ in \eqref{cbqH:def5} and use the appropriate limit transition formula.
\end{proof}

%%%%%%%%%%%%%%%%%%%%%%%%%%%%%%%%%%%%%%%%%%%%%%%%%%%%%%%%%%%%%%%%%%%%%%%%%%%%%%%%%%%
%\subsection{Terminating continuous $q^{-1}$-Hermite polynomial representations}
%%%%%%%%%%%%%%%%%%%%%%%%%%%%%%%%%%%%%%%%%%%%%%%%%%%%%%%%%%%%%%%%%%%%%%%%%%%%%%%%%%%

\noindent One has the following list of terminating ${}_4W_3^{4}$ and equivalent terminating basic hypergeometric representations for the continuous $q^{-1}$-Hermite polynomials. 

\begin{cor} 
\label{corcqiH}
Let 
$n\in\mathbb N_0$, $q\in\CCdag$,
$z\in\CCast$,
$x=\frac12(z+z^{-1})\in\CCast$.
The continuous $q^{-1}$-Hermite polynomials are given by:
\begin{eqnarray}
\label{cqiH:def1}
&&\hspace{-6.3cm}H_n(x|q^{-1})
=z^n \qphyp{1}{0}{1}{q^{-n}}{-}{q,\frac{q}{z^2}}\\
\label{cqiH:def2}
&&\hspace{-4.4cm}=\frac{q^{\binom{n}{2}}(-1)^nz^{3n}}{(z^{2};q)_n}
\Wphyp{4}{3}{4}{\frac{q^{-n}}{z^2}}{q^{-n}}{q,\frac{q^{2-n}}{z^4}}.
\end{eqnarray}
\end{cor}

\begin{proof}
Each $q^{-1}$ representation is derived from the 
corresponding representation by applying the map $q\mapsto 
1/q$ and using \eqref{poch.id:3}.
\end{proof}

%%%%%%%%%%%%%%%%%%%%%%%%%%%%%%%%%%%%%%%%%%%%%%%%%%%%%%%%%%%%%%%%%%%%%%%%%%%%%%%%%%%
\subsection{Terminating \textit{0}-parameter $q$ and $q^{-1}$ transformations}
%%%%%%%%%%%%%%%%%%%%%%%%%%%%%%%%%%%%%%%%%%%%%%%%%%%%%%%%%%%%%%%%%%%%%%%%%%%%%%%%%%%

\noindent By studying the interrelations of the terminating basic hypergeometric functions in Corollary \ref{corcqH} we obtain the following transformation result for a terminating ${}_4W_3^{-4}$.

\begin{cor}
\label{cor7.4a}
Let $n\in\mathbb N_0$, $a\in\CCast$, $q\in\CCdag$.
Then, one has 
the following transformation
formulas for a terminating 
${}_{4}W_{3}^{-4}$:
\begin{eqnarray}
&&\hspace{-1cm}
\Wphyp{4}{3}{-4}{a}{q^{-n}}{q,\frac{q^{n-2}}{a^2}}
=\frac{q^{-3\binom{n}{2}}}
{(-q^2a^2)^n}\dfrac{\left(qa,a;q\right)_n}
{(a;q)_{2n}} 
\Wphyp{4}{3}{-4}{\frac{q^{-2n}}{a}}{q^{-n}}{q,q^{5n-2}a^2}\\
&&\hspace{2.65cm}=q^{-\binom{n}{2}}\frac{\left(qa;q\right)_n}{(-qa)^n}
\qphyp{1}{0}{-1}{q^{-n}}{-}{q,\frac{1}{a}}\\
&&\hspace{2.65cm}=q^{-3\binom{n}{2}}
\frac{\left(qa; q\right)_n}{(-q^2a^2)^n}
\qphyp{1}{0}{-1}{q^{-n}}{-}{q,q^{2n}a}.
\end{eqnarray}
\end{cor}
\begin{proof}
Start with Corollary \ref{cor6.6a}, then take the limit as $c\to 0$. This
completes the proof. 
\end{proof}

%%%%%%%%%%%%%%%%%%%%%%%%%%%%%%%%%%%%%%%%%%%%%%%%%%%%%%%%%%%%%%%%%%%%%%%%%%%%%%%%%%%
%\subsection{The base ${}_4W_3$ transformations and summation}
%%%%%%%%%%%%%%%%%%%%%%%%%%%%%%%%%%%%%%%%%%%%%%%%%%%%%%%%%%%%%%%%%%%%%%%%%%%%%%%%%%%

%%%%%%%%%%%%%%%%%%%%%%%%%%%%%%%%%%%%%%%%%%%%%%%%%%%%%%%%%%%%%%%%%%%%%%%%%%%%%%%%%%%
%\subsection{Terminating \textit{0}-parameter 
 %$q^{-1}$ transformations}
%%%%%%%%%%%%%%%%%%%%%%%%%%%%%%%%%%%%%%%%%%%%%%%%%%%%%%%%%%%%%%%%%%%%%%%%%%%%%%%%%%%

\noindent By studying the interrelations of the terminating basic hypergeometric functions in Corollary \ref{corcqiH} we obtain the following transformation result for a terminating ${}_4W_3^{4}$.

\begin{cor}
\label{cor7.8}
Let $n\in\mathbb N_0$, $a\in\CCast$, $q\in\CCdag$.
%such that $|q|\ne 1$. 
Then, one has 
the following transformation formulas for a terminating ${}_{4}W_{3}^{4}$:
\begin{eqnarray}
&&\hspace{-1cm}
\Wphyp{4}{3}{4}{a}{q^{-n}}{q,q^{n+2}a^2}
=q^{5\binom{n}{2}}(-q^2a^2)^n\frac{(qa,a;q)_n}{(a;q)_{2n}}
\Wphyp{4}{3}{4}{\frac{q^{-2n}}{a}}{q^{-n}}{q,\frac{q^{2-3n}}{a^2}}\\
&&\hspace{2.8cm}=q^{2\binom{n}{2}}
\left({qa}\right)^n\left(qa;q\right)_n
\qphyp{1}{0}{1}{q^{-n}}{-}{q,\frac{q^{1-n}}{a}}\\
&&\hspace{2.8cm}=\left(qa;q\right)_n
\qphyp{1}{0}{1}{q^{-n}}{-}{q,q^{n+1}a}.
\end{eqnarray}
\end{cor}
\begin{proof}
Multiply the string of equalities in Corollary \ref{cor6.11} by
$c^n$ and then taking the limit as $c\to\infty$
completes the proof. 
Furthermore, the limit 
as $c\to\infty$
of the ${}_2\phi_{0}^{1}$'s in Corollary \ref{cor6.11} produces the same results as the limit of the ${}_1\phi_1$'s.
\end{proof}

%%%%%%%%%%%%%%%%%%%%%%%%%%%%%%%
%%%%%%%%%%%%%%%%%%%%%%%%%%%%%%%
\subsection{Terminating \textit{0}-parameter
summation formulas}
%%%%%%%%%%%%%%%%%%%%%%%%%%%%%%%
%%%%%%%%%%%%%%%%%%%%%%%%%%%%%%%

By taking the limits of the 
continuous big $q$-Hermite, $Y_n^{\pm}({\bf a}|q)$
(see Section \ref{oneparamsum}),
and the continuous big $q^{-1}$-Hermite polynomial
representations and transformations as the
parameter tends to either zero
or infinity, one can obtain 
summation formulas for ${}_{4}W_3^{-2}$, ${}_4W_3$ and 
${}_4W_{3}^{2}$.
We present these summation
formulas in Figure \ref{figure1} below.

\begin{thm}
Let $H_n(x;a|q)$ be the continuous big $q$-Hermite polynomial, $n\in\mathbb N_0$,
$a\in\CC^\ast$.
Then
\begin{eqnarray}
&&\hspace{-3.5cm}\lim_{a\to\infty}\frac{1}{a^n}H_n(x;a|q)
=\lim_{a\to 0}Y_n^-(a|q)
=
%z^{-n}
Z_n^-(q)
=q^{\binom{n}{2}}(-1)^n,\\
&&\hspace{-3.5cm}\lim_{a\to\infty}\frac{1}{a^n}Y_n^-(a|q)
=\lim_{a\to0}Y_n^+(a|q)
=Z_n(q)
=\delta_{n,0},\\
&&\hspace{-3.5cm}\lim_{a\to\infty}\frac{1}{a^n}
Y_n^+(a|q)=
\lim_{a\to \infty}\frac{1}{a^n}
H_n(x;a|q^{-1})
=
%q^{-3\binom{n}{2}}
%(-1)^n
Z_n^+(q)
=q^{-\binom{n}{2}}(-1)^n,
\end{eqnarray}
%\[
%\lim_{a_p\to\infty}
%\frac{1}{a_p^n}Q_n(x;{\bf a}|q^{-1})
%=q^{-3\binom{n}{2}}(-a_{12})^n
%%Y_n^-({\bf a}_{[p]}^{-1}|q)
%=q^{-2\binom{n}{2}}a_{12}^n(a_{12}^{-1};q)_n,
%\]
where
\begin{eqnarray}
&&\hspace{-5.2cm}Z_n^-(q):=
\frac{q^{\binom{n}{2}}(-1)^n
%z^n
}{(z^{-2};q)_n}
\Wphyp{4}{3}{-2}{q^{-n}z^2}{q^{-n}}{q,\frac{q^{2n-1}}{z^2}},\\
&&\hspace{-5.2cm}Z_n(q):=\frac{q^{2\binom{n}{2}}z^{-n}}{(z^{-2};q)_n}
\Whyp43{\frac{q^{-n}}{z^2}}{q^{-n}}{q,q^n},\\
&&\hspace{-5.2cm}Z_n^+(q):=
%q^{3\binom{n}{2}}
\frac{
%(-1)^n
z^{-2}}{(z^{-2};q)_n}
\Wphyp{4}{3}{2}{q^{-n}z^2}{q^{-n}}{q,qz^{2}},
\end{eqnarray}
and $Y_n^{\pm}$ are defined in 
\eqref{Ynminform},
\eqref{Ynplusform}.
\end{thm}
\begin{proof}
Applying the necessary
limits using Lemma \ref{leminftyWhyp} 
completes the proof. 
\end{proof}

\begin{cor}
\label{cor7.3}
Let $n\in\mathbb N_0$, $a\in\CCast$, $q\in\CCdag$. 
%such that $|q|\ne 1$.
Then, one has 
the following transformation
and summation 
formulas for a terminating 
${}_{4}W_3^{-2}$:
\begin{eqnarray}
&&\Wphyp{4}{3}{-2}{a}{q^{-n}}{q,\frac{q^{n-1}}{a}}
=q^{-\binom{n}{2}}\left(\frac{-1}{qa}\right)^n
\frac{(qa,a;q)_n}
{(a;q)_{2n}}
\Wphyp{4}{3}{-2}
{\frac{q^{-2n}}{a}}{q^{-n}}
{q,q^{3n-1}a}\nonumber
\\&&\hspace{4.2cm}
=q^{-\binom{n}{2}}
\left(\frac{-1}{qa}\right)^n(qa;q)_n.
\end{eqnarray}
\end{cor}
\begin{proof}
Starting with Corollary \ref{cor6.6a} and letting 
$c\to\infty$ using Lemma \ref{leminftyWhyp}
completes the proof. 
One may produce the identical 
result by starting with Corollary \ref{cor6.4},
and taking the limit as $c\to0$.
This is indicated in 
Figure \ref{figure1} below in the box for
$Z_n^{-}$ which is given in terms of a terminating
${}_{4}W_3^{-2}$.
\end{proof}

\begin{cor}[{\cite[(2.3.4)]{GaspRah}}]
\label{cor7.3a}
Let $n\in\mathbb N_0$, $a\in\CCast$, 
$q\in\CCdag$.
Then, one 
has the following transformation and summation formulas for a terminating 
${}_{4}W_3$:
\begin{eqnarray}
&&\Whyp{4}{3}{a}{q^{-n}}{q,q^n}
=q^{\binom{n}{2}}(-1)^n
\frac{(qa,a;q)_n}
{(a;q)_{2n}}
\Whyp{4}{3}
{\frac{q^{-2n}}{a}}{q^{-n}}
{q,q^n}=\delta_{n,0}.
\end{eqnarray}
\end{cor}
\begin{proof}
Starting with Corollary \ref{cor6.4} and letting 
$c\to\infty$ using Lemma \ref{leminftyWhyp}
completes the proof. 
One may produce the identical result by
starting with Corollary \ref{cor6.8},
and taking the limit as $c\to0$.
This is indicated in Figure \ref{figure1} below in the box for
$Z_n$ which is given in terms of a terminating ${}_{4}W_3$.
\end{proof}

\begin{cor}
\label{cor7.7}
Let $n\in\mathbb N_0$, $a\in\CCast$, 
$q\in\CCdag$.
%such that $|q|\ne 1$.
Then, one has 
the following transformation
and summation 
formulas for a terminating 
${}_{4}W_{3}^{2}$:
\begin{eqnarray}
&&\hspace{-0.7cm}\Wphyp{4}{3}{2}{a}{q^{-n}}{q,q^{n+1}a}
=q^{3\binom{n}{2}}(-qa)^n
\frac{(qa,a;q)_n}
{(a;q)_{2n}}
\Wphyp{4}{3}{2}
{\frac{q^{-2n}}{a}}{q^{-n}}
{q,\frac{q^{1-n}}{a}}=(qa;q)_n.
\end{eqnarray}
\end{cor}
\begin{proof}
Starting with Corollary \ref{cor6.8} and letting 
$c\to\infty$ using Lemma \ref{leminftyWhyp}
completes the proof.
One may produce the identical result by
starting with Corollary \ref{cor6.11},
and taking the limit as $c\to0$.
This is indicated in 
Figure \ref{figure1} below in the box for
$Z_n^{+}$ which is given in terms of a terminating
${}_{4}W_{3}^{2}$.
\end{proof}

\section{The symmetry group \textit{q}-Askey scheme}
\label{symmetrygroup}

We provide a diagram describing the hierarchy of transformation and summation formulas we have derived through the $q$ and $q^{-1}$-symmetric scheme (see Figure \ref{figure1}). 
The terminating transformations referred
to in this paper 
should match to
a \hbox{($q$-)Askey} scheme of symmetry
groups, parallel to the \hbox{($q$-)Askey} scheme of
(basic) hypergeometric orthogonal polynomials.
This data we have presented is critical for summarizing the properties of
this transformation group scheme.

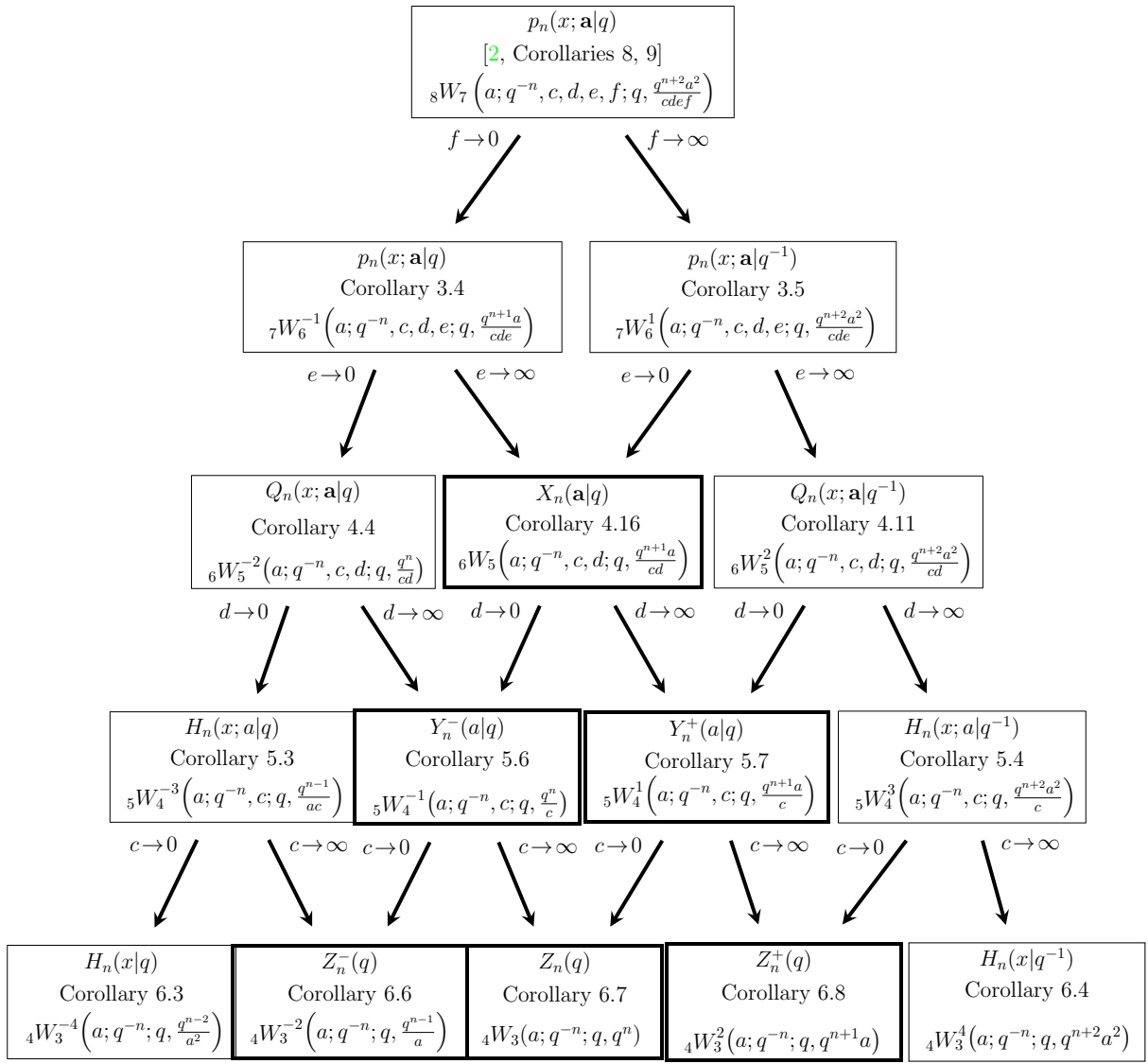
\begin{figure}[!hbt]
%\begin{sidewaysfigure}
%\label{sidewaysfigure}
%\begin{figure}[!htb]
\caption{
A summation (bold outline boxes) and transformation (regular 
outline boxes) scheme for the $q$-symmetric and $q^{-1}$-symmetric 
subfamilies of basic hypergeometric orthogonal polynomials in the 
$q$-Askey scheme.
\vspace*{4mm}
%\\[0.4cm]
}
\begin{tikzpicture}[level distance=.05cm,sibling distance=.04cm,scale=0.57
,every node/.style={scale=0.62},inner sep=11pt]
\node(cqH)
at (0.0,0.0){
\setlength{\fboxsep}{4pt}
\fbox{
$\begin{array}{cc}
\text{\Large$H_n(x|q)$}
\\[6.5pt]\text{\Large Corollary \ref{cor7.4a}}\\[3.5pt]
\text{\Large$
%{}_{4}W_{3}^{-4}
\hspace{-0.2cm}\Wphyp{4}{3}{-4}{a}{q^{-n}}{q,\frac{q^{n-2}}{a^2}}\hspace{-0.2cm}
$}\\[0.1cm]
\end{array}$
}
};

\node(Zm)
at (5.7,0.0){
{\setlength{\fboxrule}{0.070cm}
\fbox{
$\begin{array}{cc}
\text{\Large$Z^{-}_n(q)$}
\\[6.5pt]\text{\Large Corollary \ref{cor7.3}}\\[3.5pt]
\text{\Large
\hspace{-0.2cm}$\Wphyp{4}{3}{-2}{a}{q^{-n}}{q,\frac{q^{n-1}}{a}}\hspace{-0.2cm}$
}
\end{array}$
}
}
}; 

\node(Z)
at (11,0.0){
{\setlength{\fboxrule}{0.070cm}
\fbox{
$\begin{array}{cc}
\text{\Large$Z_n(q)$}
\\[7.0pt]
\text{\Large Corollary \ref{cor7.3a}}\\[11.0pt]
\text{\Large
\hspace{-0.2cm}$\Whyp{4}{3}{a}{q^{-n}}{q,q^n}$
\hspace{-0.30cm}
%HSC%\\[0.1cm]
}\\[0.1cm]
\end{array}$
}
}
};

\node(Zp)
at (16.4,0.0){
{\setlength{\fboxrule}{0.070cm}
\fbox{
$\begin{array}{cc}
\text{\Large$Z^{+}_n(q)$}
\\[7.2pt]
\text{\Large Corollary \ref{cor7.7}}\\[11.2pt]
\text{\Large
\hspace{-0.2cm}$\Wphyp{4}{3}{2}{a}{q^{-n}}{q,q^{n+1}a}$
\hspace{-0.33cm}
%HSC%\\[0.1cm]
}\\[0.0cm]
\end{array}$
}
}
};

\node(cqiH)
at (22.4,0.0){
{\setlength{\fboxrule}{.0225cm}
\fbox{
$\begin{array}{cc}
\text{\Large$H_n(x|q^{-1})$}
\\[6.0pt]
\text{\Large Corollary \ref{cor7.8}}\\[11.6pt]
\text{\Large
\hspace{-0.2cm}$\Wphyp{4}{3}{4}{a}{q^{-n}}{q,q^{n+2}a^2}\hspace{-0.28cm}$
%HSC%\\[0.1cm]
}\\[0.117cm]
\end{array}$
}
}
};

%%%%% end of bottom level %%%%%
\node(cbqH) 
at (2.8,5.8*1.0){
\fbox{
$\begin{array}{cc}
\text{\Large$H_n(x;a|q)$}
\\[6.5pt]\text{\Large Corollary \ref{cor6.6a}}\\[3.5pt]
\text{\Large$
\hspace{-0.2cm}\Wphyp{5}{4}{-3}{a}{q^{-n},c}{q,\frac{q^{n-1}}{ac}}\hspace{-0.2cm}
$}
\end{array}$
}
};

\node(Ym) 
at (8.6,5.8*1.0){
{\setlength{\fboxrule}{0.070cm}
\fbox{
$\begin{array}{cc}
\text{\Large$Y^{-}_n(a|q)$}
\\[6.5pt]\text{\Large Corollary \ref{cor6.4}}\\[9.0pt]
\text{\Large$
\hspace{-0.2cm}\Wphyp{5}{4}{-1}{a}{q^{-n},c}{q,\frac{q^n}{c}}\hspace{-0.2cm}
$}
\end{array}$
}
}
};

\node(Yp) 
at (14.5,5.8*1.0){
{\setlength{\fboxrule}{0.070cm}
\fbox{
$\begin{array}{cc}
\text{\Large$Y^{+}_n(a|q)$}
\\[6.5pt]\text{\Large Corollary \ref{cor6.8}}\\[1.0pt]
\text{\Large$
%{}_5W_{4,1}
\hspace{-0.2cm}\Wphyp{5}{4}{1}{a}{q^{-n},c}{q,\frac{q^{n+1}a}{c}}\hspace{-0.2cm}
$}
\end{array}$
}
}
};

\node(cbqiH) 
at (20.8,5.8*1.0){
\fbox{
$\begin{array}{cc}
\text{\Large$H_n(x;a|q^{-1})$}
\\[6.5pt]\text{\Large Corollary \ref{cor6.11}}\\[3.5pt]
\text{\Large$
%{}_5W_{4,3}
\hspace{-0.2cm}\Wphyp{5}{4}{3}{a}{q^{-n},c}{q,\frac{q^{n+2}a^2}{c}}\hspace{-0.2cm}
$}
\end{array}$
}
};
%%%%%% second level %%%%%%%%%
\node(ASC) 
at (4.8,5.8*2.0){
\fbox{
$\begin{array}{cc}
\text{\Large$Q_n(x;{\bf a}|q)$}
\\[9.8pt]\text{\Large Corollary \ref{cor:5.8}}\\[7.8pt]
\text{\Large$
%{}_{6}W_{5}^{-2}
\hspace{-0.2cm}\Wphyp{6}{5}{-2}{a}{q^{-n},c,d}{q,\frac{q^{n}}{cd}}\hspace{-0.2cm}
$}
\end{array}$
}
};

\node(X) 
at (2.8*4.00,5.8*2.00){
{\setlength{\fboxrule}{0.070cm}
\fbox{
$\begin{array}{cc}
\text{\Large$X_n({\bf a}|q)$}
\\[5.5pt]\text{\Large Corollary \ref{cor5.6}}\\[2.5pt]
\text{\Large$
%{}_6W_5
\hspace{-0.2cm}\Whyp{6}{5}{a}{q^{-n},c,d}{q,\frac{q^{n+1}a}{cd}}\hspace{-0.2cm}
$}
\end{array}$
}
}
};

\node(qiASC) 
at (18,5.8*2.00){
\fbox{
$\begin{array}{cc}
\text{\Large$Q_n(x;{\bf a}|q^{-1})$}
\\[6.6pt]\text{\Large Corollary \ref{cor5.14}}\\[3.6pt]
\text{\Large$
%{}_6W_{5,2}
\hspace{-0.2cm}\Wphyp{6}{5}{2}{a}{q^{-n},c,d}{q,\frac{q^{n+2}a^2}{cd}}\hspace{-0.2cm}
$}
\end{array}$
}
};

\node(cdqH) 
at (2.8*2.50,5.8*3.00){
\fbox{
$\begin{array}{cc}
\text{\Large$p_n(x;{\bf a}|q)$}
%, Corollary \ref{cor:4.3}}
\\[6.5pt]\text{\Large Corollary \ref{cor:4.3}}\\[3.5pt]
\text{\Large$\Wphyp{7}{6}{-1}{a}{q^{-n},c,d,e}{q,\frac{q^{n+1}a}{cde}}$ }
\end{array}$
}
};

\node(cdqiH) 
at (2.8*5.50,5.8*3.00){
\fbox{
$\begin{array}{cc}
\text{\Large$p_n(x;{\bf a}|q^{-1})$}
\\[6pt]\text{\Large Corollary \ref{cor:4.13}}\\[3.5pt]
\text{\Large$\Wphyp{7}{6}{1}{a}{q^{-n},c,d,e}{q,\frac{q^{n+2}a^2}{cde}}$}
\end{array}$
}
};

\node(AW) 
at (2.8*4.00,5.8*4.00){
\fbox{
$\begin{array}{cc}
\text{\Large$p_n(x;{\bf a}|q)$}
\\[6.5pt]\text{\Large \cite[Corollaries 8, 9]{CohlCostasSantos20b}}\\[3.5pt]
\text{\Large${}_8W_7\left(a;q^{-n},c,d,e,f;q,\frac{q^{n+2}a^2}{cdef}\right)$}
\end{array}$
}
};
\draw[-stealth,line width=1.5pt] (AW)--(cdqH); 
\draw[-stealth,line width=1.5pt] (AW)--(cdqiH);
\draw[-stealth,line width=1.5pt] (cdqH)--(ASC);
\draw[-stealth,line width=1.5pt] (cdqH)--(X);
\draw[-stealth,line width=1.5pt] (cdqiH)--(X);
\draw[-stealth,line width=1.5pt] (cdqiH)--(qiASC);
\draw[-stealth,line width=1.5pt] (ASC)--(cbqH);
\draw[-stealth,line width=1.5pt] (ASC)--(Ym);
\draw[-stealth,line width=1.5pt] (X)--(Ym);
\draw[-stealth,line width=1.5pt] (X)--(Yp);
\draw[-stealth,line width=1.5pt] (qiASC)--(Yp);
\draw[-stealth,line width=1.5pt] (qiASC)--(cbqiH);
\draw[-stealth,line width=1.5pt] (cbqH)--(cqH);
\draw[-stealth,line width=1.5pt] (cbqH)--(Zm);
\draw[-stealth,line width=1.5pt] (Ym)--(Zm);
\draw[-stealth,line width=1.5pt] (Ym)--(Z);
\draw[-stealth,line width=1.5pt] (Yp)--(Z);
\draw[-stealth,line width=1.5pt] (Yp)--(Zp);
\draw[-stealth,line width=1.5pt] (cbqiH)--(Zp);
\draw[-stealth,line width=1.5pt] (cbqiH)--(cqiH);

\path(AW)->(cdqH) node[draw=none,pos=0.05,left=0.1pt]{\Large$f\!\to\!0$};
\path(AW)->(cdqiH) node[draw=none,pos=0.05,right=0.1pt]{\Large$f\!\to\!\infty$};

\path (cdqH)->(ASC) node[draw=none, pos=0.05,left=0.1pt]{\Large$e\!\to\! 0$};
\path (cdqH)->(X) node[draw=none, pos=0.05,right=0.1pt]{\Large$e\!\to\! \infty$};
\path (cdqiH)->(X) node[draw=none, pos=0.05,left=0.1pt]{\Large$e\!\to\!0$};
\path (cdqiH)->(qiASC) node[draw=none, pos=0.05,right=0.1pt]{\Large$e\!\to\!\infty$};

\path (ASC)->(cbqH) node[draw=none, pos=0.05,left=0.1pt]{\Large$d\!\to\!0$};
\path (ASC)->(Ym) node[draw=none,pos=0.065,right=0.1pt]{\Large$d\!\to\!\infty$};
\path (X)->(Ym) node[draw=none,pos=0.05,left=0.1pt]{\Large$d\!\to\!0$};
\path (X)->(Yp) node[draw=none,pos=0.05,right=0.1pt]{\Large$d\!\to\!\infty$};
\path (qiASC)->(Yp) node[draw=none,pos=0.05,left=0.1pt]{\Large$d\!\to\!0$};
\path (qiASC)->(cbqiH) node[draw=none,pos=0.05,right=0.1pt]
{\Large$d\!\to\!\infty$};

\path (cbqH)->(cqH) node[draw=none, pos=0.05,left=0.1pt]{\Large$c\!\to\! 0$};
\path (cbqH)->(Zm) node[draw=none, pos=0.075,right=0.1pt]{\Large$c\!\to\! \infty$};
\path (Ym)->(Zm) node[draw=none, pos=0.05,left=0.1pt]{\Large$c\!\to\! 0$};
\path (Ym)->(Z) node[draw=none, pos=0.075,right=0.1pt]{\Large$c\!\to\! \infty$};
\path (Yp)->(Z) node[draw=none, pos=0.05,left=0.1pt]{\Large$c\!\to\! 0$};
\path (Yp)->(Zp) node[draw=none, pos=0.065,right=0.1pt]{\Large$c\!\to\!\infty$};
\path (cbqiH)->(Zp) node[draw=none, pos=0.075,left=0.0pt]{\Large$c\!\to\!0$};
\path (cbqiH)->(cqiH) node[draw=none, pos=0.05,right=0.1pt]{\Large$c\!\to\!\infty$};
\end{tikzpicture}
\label{figure1}
%\end{sidewaysfigure}
\end{figure}

As explained in \cite{CohlCostasSantos20b}, the total 
number of transformations of a balanced terminating ${}_4\phi_3$ which corresponds to the Askey--Wilson polynomials is 720. As explained in \cite{VanderJeugtRao1999}, this corresponds to the number of elements of the symmetric group $S_6$ with order $|S_6|=720$. In \cite[Proposition 16]{CohlCostasSantos20b}, it was shown that the total number of transformations of the terminating very-well-poised ${}_8W_7$ is the same number and also corresponds to the symmetric group $S_6$. 

When considering the continuous dual $q$ and $q^{-1}$-Hahn polynomials, these correspond to terminating ${}_3\phi_2$ and according to \cite{VanderJeugtRao1999} (see also \cite{Raoetal1992}), the symmetry group of these functions, a subgroup of $S_6$, is a nonsimple group of order 72 which can be described in terms of extended symmetries of the hexagon. Examining Corollary \ref{cor:4.3} one can see that the total number of permutations and rearrangements for the ${}_7W_6^{-1}$'s given in \eqref{cor4.3:r1}--\eqref{cor4.3:r10x} is 36 $(6\times 6)$ and for \eqref{cor4.3:r3} which is a ${}_7W_6^1$ is also 36, then combined together we obtain 72. The same is true for Corollary \ref{cor:4.13}. One can see that the total number of permutations and rearrangements for the ${}_7W_6^1$'s given in \eqref{cor:4.13-0}--\eqref{cor:4.13-10x} is 36 $(6\times 6)$ and for \eqref{cor:4.13-2} which is a ${}_7W_6^{-1}$ is also 36. Then combined together we obtain 72. So the symmetry group of each grouping corresponds to $|S_3\times S_3|=36$ (see \cite[\S IV]{VanderJeugtRao1999}). One might consider that this symmetry group is $(S_3\times S_3) \rtimes C_2$, where $C_2$ is the cyclic group of order two and $\rtimes$ is the semi-direct product. This group is a subgroup of $S_6$ and has order $72$.

If one considers the transformation groups of the terminating $q$ and $q^{-1}$-Al-Salam--Chihara polynomials, then we can see that both of these correspond to a terminating ${}_2\phi_1$. The symmetry group of the nonterminating ${}_2\phi_1$ has is known to correspond to the dihedral group of order 12, $D_6$. However for the terminating ${}_2\phi_1$'s which we are encountering the symmetry group must be a subgroup of $S_6$. Examining Corollary \ref{cor:5.8} one can see that the total number of permutations and rearrangements for the ${}_6W_5^{-2}$'s given in \eqref{cor:5.8.1}--\eqref{cor:5.8.2} is 4 $(2\times 2)$ and for \eqref{cor:5.8.3} which is a ${}_6W_4^2$ is also 4, then combined together we obtain 8. 
The same is true for Corollary \ref{cor5.14b}. One can see that the total number of permutations and rearrangements for the ${}_6W_5^2$'s given in \eqref{cor5.14.1}--\eqref{cor5.14.2} is 4 $(2\times 2)$ and for \eqref{cor5.14.3} which is a ${}_6W_5^{-2}$ is also 4. Then combined together we obtain 8. 
In this case we conjecture that the symmetry group of the terminating ${}_2\phi_1$ which corresponds to the 2-variable symmetric polynomials in the $q$-Askey scheme corresponds to $|D_4|=8$ which is known to be a subgroup of $S_6$. One should also consider the rearrangements and permutations of \eqref{cor:5.8.4}--\eqref{cor:5.8.11}, and as well for \eqref{cor5.14.9}--\eqref{cor5.14.11}. One can see that for each of these equations the numbers are 2 for a total of 16. So these must break down into two separate groupings each. 

In the one parameter case which corresponds to the continuous big $q$ and big $q^{-1}$-Hermite polynomials and the terminating ${}_5W_4^{-3}$ and ${}_5W_4^3$, the orders are 2 since there are two separate representations each with one one possible rearrangement for each. In the zero parameter case which corresponds to the continuous $q$ and $q^{-1}$-Hermite polynomials and the terminating ${}_4W_3^{-4}$ and ${}_4W_3^4$, the orders are 2 as there are two representations only one possible rearrangement for each. It is not yet clear what the symmetry groups for these might be.

\section*{Acknowledgments}
H. S. C.~would like to thank Mourad Ismail, 
Tom Koornwinder, 
Eric Rains,
Michael Schlosser
and Joris Van der Jeugt 
for valuable discussions. 
Also thanks to Tom Koornwinder for 
communicating to HSC that the transformations 
referred to in this paper 
should match to
a ($q$-)Askey scheme of symmetry
groups, parallel to the ($q$-)Askey scheme of
(basic) hypergeometric orthogonal polynomials. 
%\newpage
%%%%\section*{References}

%\bibliographystyle{plain}
%\bibliography{refbib}

%%\bibliography{../refbib} % Roberto bib
%\bibliography{/home/hcohl/tex/refbib} % Howard bib yes!

\def\cprime{$'$} \def\dbar{\leavevmode\hbox to 0pt{\hskip.2ex \accent"16\hss}d}

\end{document}